\documentclass[final, 12pt,reqno]{amsart}
\usepackage{amssymb,amsmath,graphicx,amsfonts,euscript}
\usepackage{color}
\usepackage[pagewise]{lineno}
\usepackage{showkeys}
\usepackage{fancyhdr}
\usepackage[margin=1in]{geometry}
\usepackage{titletoc}
\usepackage{appendix}
\usepackage{color}
\usepackage{hyperref}

\newtheorem{theorem}{\bf Theorem}[section]

\newtheorem{proposition}[theorem]{\bf Proposition}

\newtheorem{remark}{\bf Remark}[section]

\newtheorem{lemma}[theorem]{\bf Lemma}

\numberwithin{equation}{section}
\allowdisplaybreaks[4]

\newcommand{\beq}{\begin{equation}}
	\newcommand{\eeq}{\end{equation}}
\newcommand{\ben}{\begin{eqnarray}}
	\newcommand{\een}{\end{eqnarray}}
\newcommand{\beno}{\begin{eqnarray*}}
	\newcommand{\eeno}{\end{eqnarray*}}

\title[Stability threshold for 2-D Poiseuille flow]{Enhanced dissipation and stability of Poiseuille flow  for two-dimensional Boussinesq system}
\author[Gaofeng Wang  ]{\sc Gaofeng Wang$^{1}$}

\address{$^1$ School of Mathematical Sciences , Shanghai Jiao Tong University, Shanghai, 200240,  P.R.China.}
\email{1462047796@qq.com}

\begin{document}
	%\linenumbers
	\bibliographystyle{abbrv}
	\maketitle
	\vskip .01in
	\begin{center}
		\sc Abstract
	\end{center}
	
	\vskip .05in
	
	 We investigate the nonlinear stability problem  for the two-dimensional Boussinesq system around the Poiseuille flow in a finite channel. The system  has the characteristic of  Navier-slip boundary condition for the velocity and Dirichlet boundary condition for the temperature, with a small viscosity $\nu$ and small thermal diffusion $\mu,$ respectively.  More precisely, we prove that if the initial  velocity and initial  temperature satisfies$$||u_{0}-(1-y^2,0) ||_{H^{\frac{7}{2}+}}\leq c_0\min\left\lbrace \mu,\nu\right\rbrace ^{\frac{2}{3}}$$
	 and 
	 $$ ||\theta_{0}||_{H^1}+|||D_x|^{\frac{1}{8}}\theta_{0}||_{H^1}\leq c_1\min\left\lbrace \mu,\nu\right\rbrace ^{{\frac{31}{24}}}$$ for some small constants $c_0$ and $c_1$  which are both
	independent of $\mu,\nu$, then we can reach the conclusion that the velocity remains within $O\left( \min\left\lbrace \mu,\nu\right\rbrace ^{\frac{2}{3}}\right) $ of  the Poiseuille  flow; the temperature remains $O\left( \min\left\lbrace \mu,\nu\right\rbrace ^{\frac{31}{24}}\right) $ of the constant 0, and approaches to 0 as $t\rightarrow\infty.$

	\maketitle
	
	%\setcounter{tocdepth}{1}
	%\tableofcontents
	\titlecontents{section}[0pt]{\vspace{0\baselineskip}\bfseries}
	{\thecontentslabel\quad}{}%
	{\hspace{0em}\titlerule*[10pt]{$\cdot$}\contentspage}
	
	\titlecontents{subsection}[1em]{\vspace{0\baselineskip}}
	{\thecontentslabel\quad}{}%
	{\hspace{0em}\titlerule*[10pt]{$\cdot$}\contentspage}
	
	\titlecontents{subsubsection}[2em]{\vspace{0\baselineskip}}
	{\thecontentslabel\quad}{}%
	{\hspace{0em}\titlerule*[10pt]{$\cdot$}\contentspage}

	%\setcounter{tocdepth}{2}
	%\tableofcontents

	\vskip .3in

	
	\section{Introduction}
The Boussinesq system embodies the fundamental principles of physics governing fluid dynamics driven by buoyancy. It stands as one of the most widely employed models, universally applicable in atmospheric and oceanic flows, and plays a central role in Rayleigh-Bénard convection studies  \cite{ref1,ref2,ref3,ref4}. Henceforth, it becomes evident that the Boussinesq equation holds significant mathematical significance. The two-dimensional Boussinesq equation serves as a reduced-order model for three-dimensional fluid dynamics equations. Notably, certain key characteristics of the 3D Euler and Navier-Stokes equations are preserved within the 2D Boussinesq equations, including the vortex stretching mechanism. The inviscid 2D Boussinesq equations can be identified as Euler equations for axisymmetric swirling flows in three dimensions  \cite{ref49}. Moreover, the Boussinesq equations possess distinctive features of their own and offer ample opportunities for novel discoveries.

Due to its broad range of physical applications and profound mathematical implications, the Boussinesq equations have recently garnered significant attention. The development of the mathematical theory of Boussinesq equations has been driven by two fundamental problems: the global regularity problem and the stability problem. In this discussion, we will primarily focus on stability issues. Our research aims to investigate the stability of solutions under perturbations near Poiseuille flows and their long-time behavior.
In this article, we first study the 2D  Navier-Stokes Boussinesq system in a domain  
$$\Omega=\left\lbrace (x,y): x\in  \mathbb{T},y\in(-1,1)\right\rbrace ,$$
given by 
\begin{equation}\label{1.1}
	\left\{
	\begin{aligned}
		&\partial _tv+v\cdot\nabla v-\nu \Delta v+\nabla q=\theta e_2,\\
		&\partial_t\theta+v\cdot\nabla \theta-\mu \Delta \theta=0,\\
		&\nabla\cdot v=0,\\
		&\theta(0,x,y)=\theta_{0}(x,y),v(0,x,y)=v_{0}(x,y),
	\end{aligned}
	\right.
\end{equation}
where $\nu$, $\mu$  donote the viscosity coefficient and   the thermal diffusivity, respectively. $v(t,x,y)=(v^1(t,x,y),v^2(t,x,y))$ is the two-dimensional velocity field, $q(t,x,y)$ is the pressure, $\theta$ is the temperature, $e_2=(0,1)$ is the unit vector in the vertical direction.
The first equation in \eqref{1.1} represents the incompressible Navier-Stokes equation with vertical buoyancy forcing, while the second equation describes the balance between temperature convection and diffusion. It is well-known that the 2D Navier-Stokes Boussinesq equation is globally well-posed. For further details, please refer to \cite{ref31, ref30, ref32, ref33} and references therein.

In this paper, our objective is to comprehend the stability and long-time behavior of perturbations in proximity to the Poiseuille flow, 
$$U_s=(1-y^2,0),\quad \theta_s=0,\quad \nabla Q=(Q_0,0) ~\text{for some constant } ~Q_0$$
which represents a steady state solution of equation \eqref{1.1}. 

Subsequently, we introduce the perturbation
$$u=v-(1-y^2,0),\quad \tilde{\theta}= \theta,\quad p=q-Q,$$ which satisfies
\begin{equation}\label{1.2}
	\left\{
	\begin{aligned}
		&\partial _tu+(1-y^2)\partial_x u+	\begin{pmatrix}
			-2y	u^2\\
			0
		\end{pmatrix}+u\cdot\nabla u-\nu \Delta u+\nabla p=	\begin{pmatrix}
			0\\
			\tilde{\theta}
		\end{pmatrix},\\
		&\partial_t\tilde{\theta}+u\cdot\nabla\tilde{\theta}+(1-y^2)\partial_x\tilde{\theta}-\mu \Delta \tilde{\theta}=0,\\
		&\nabla\cdot u=0,\\
		&\tilde{\theta}(0,x,y)=\tilde{\theta}_{0}(x,y),u(0,x,y)=u_{0}(x,y).\\
	\end{aligned}
	\right.
\end{equation}
For the sake of convenience, we shall denote  $\theta$ as $\tilde{\theta}$. The Navier-slip condition for the perturbation of velocity  $u$  is taken into consideration in order to mitigate the boundary layer effect. Specifically, the Dirichlet boundary condition for temperature $\theta$ is imposed, satisfying
\begin{equation}\label{1.3}
	\left\{
	\begin{aligned}
		&\partial _tu+(1-y^2)\partial_x u+	\begin{pmatrix}
			-2yu^2\\
			0
		\end{pmatrix}+u\cdot\nabla u-\nu \Delta u+\nabla p=\begin{pmatrix}
			0\\
			\theta
		\end{pmatrix},\\
		&\partial_t\theta+(1-y^2)\partial_x \theta +u\cdot\nabla \theta-\mu \Delta \theta=0,\\
		&\theta(0,x,y)=\theta_{0}(x,y),u(0,x,y)=u_{0}(x,y),\\
		&	\begin{pmatrix}
			\partial_y	u^1(t,x,\pm1)\\
			u^2(t,x,\pm1)
		\end{pmatrix}
		=\begin{pmatrix}
			0\\
			0
		\end{pmatrix},\quad \theta( t,x,\pm1)=0.
	\end{aligned}
	\right.
\end{equation}
Due to the existence of the pressure term in system$ \eqref{1.3}_1$, there are obstacles to deal with the velocity field directly. 
By takng the curl $(\nabla\times )$ on both sides of $\eqref{1.3}_1$, not only can the pressure term be eliminated, but also the vector equation can be transformed into a scalar equation.
The vorticity $\omega$ and the
stream function $\varphi$	 are defined by
$$\omega=\nabla\times u=\partial_y u^1-\partial_x u^2,\quad u=\nabla^{\perp}\varphi=(\partial_y\varphi,-\partial_x \varphi),\quad \Delta \varphi=\omega,$$which satisfies
\begin{equation}\label{1.4}
	\left\{
	\begin{aligned}
		&\partial _t\omega+(1-y^2)\partial_x\omega +2\partial_x\Delta^{-1}\omega+u\cdot\nabla \omega-\nu \Delta \omega=-\partial_x\theta ,\\
		&\partial_t\theta+(1-y^2)\partial_x \theta +u\cdot\nabla \theta-\mu \Delta \theta=0,\\
		&\theta(0,x,y)=\theta_{0}(x,y),u(0,x,y)=u_{0}(x,y),\\
		&\omega(t,x,\pm1)=0,\quad\theta(t,x,\pm1)=0.
	\end{aligned}
	\right.
\end{equation}
The stability analysis of shear flow will be presented after a comprehensive review of relevant findings.

Before stating our main result,  let us first review previous studies about the stability of steady states. The system of \eqref{1.1} has a series of steady-state solutions.
The first steady-state  of the system \eqref{1.1} is the hydrostatic equilibrium 
\begin{equation}\label{1.5}
	u_{he}=0,\quad\theta_{he}=\beta y,\quad P_{he}=\frac{1}{2}\beta y^2,
\end{equation}
which is a prominent topic in fluid dynamics and astrophysics. Remarkable progress \cite{ref5,ref6,ref7} has been made on the stability and  large-time behavior  of the steady-state \eqref{1.5}. Second, the study of steady-state about shear flow has a long history,  the earliest of which can be traced back to Taylor \cite{T1931}, Goldstein \cite{G1931} and Synge \cite{Synge1933}. 
There are many kinds of steady-states about shear flow of the system  \eqref{1.1}, such as the velocity field is the Couette flow and the temperature is a linear function of the vertical component, namely,
\begin{equation}\label{1.6}
	u_{s}=(y,0),\quad\theta_{s}=-\gamma^2 y+1,\quad P_{s}=\frac{1}{2}\gamma^2 y^2-y+C,
\end{equation}
where the constant $\gamma^2$ is the Richardson number. The Miles-Howard theorem \cite{H1961} makes sure  that  any  flow  in  the  inviscid  non-diffusive  limit  is  linearly stable if the local Richardson number exceeds the value $\dfrac{1}{4}$ everywhere. Zhai and Zhao \cite{Zhao} proved the nonlinear asymptotic stability of the steady-state \eqref{1.6} when the Richardson number $\gamma ^2\geq\dfrac{1}{4}$. The system  \eqref{1.1} has another steady-state solution when the Richardson number $\gamma ^2=0,$ namely 
\begin{equation}\label{1.7}
	u_{s}=(y,0),\quad\theta_{s}=1,\quad P_{s}=-y+C.
\end{equation}
There were quantities of works \cite{ref8,ref9,ref54,ref10,ref11,ref12} on the stability of the steady-state \eqref{1.7}  in the infinite channel case $\mathbb{T}\times\mathbb{R}.$  The best  result concerning  stability threshold is 
$$||\omega_{0}||_{H^s}\leq c_0\nu^{\frac{1}{3}},\quad||\theta_{0}||_{H^s}\leq c_0\nu^{\frac{5}{6}},\quad|||D_x|^\frac{1}{3}\theta_{0}||_{H^s}\leq c_0\nu^{\frac{2}{3}},$$
with $s>7$  when $\nu=\mu$, which was proved by Zhang and Zi \cite{ref54}. In their paper, they showed that the enhanced dissipation and non-viscous damping generated by the linear non-self-adjoint operator $y\partial_x-\nu\Delta$ in the perturbation equation prompted nonlinear stability possible.
It is well-worth mentioning that Zhao et al. \cite{ref13} proved the following stability result for  initial larger perturbations, namely
$$||\omega_{0}||_{H^s}\leq c_0\min\left\lbrace \nu,\mu\right\rbrace ^{\frac{1}{3}},\quad||\theta_{0}||_{H^s}\leq c_0\min\left\lbrace \nu,\mu\right\rbrace ^{\alpha},$$
with some s large and $\alpha\geq\frac{5}{6} $ by using time-dependent Fourier multiplier  when $\nu\neq\mu$.
%The system \eqref{1.1} for the steady-state \eqref{1.7}   is also  well studied in a finite channel case $\mathbb{T}\times \text{I}$ with non-slip boundary condition on the velocity. We can refer to \cite{ref14} and  the stability  threshold is 
%$$||\omega_{0}||_{H^1}\leq c_0\min\left\lbrace \nu,\mu\right\rbrace ^{\frac{1}{2}},\quad||\theta_{0}||_{H^1}+\left\||D_x|^{\frac{1}{6}}\theta_{0}\right\|_{H^1}\leq c_0\min\left\lbrace \nu,\mu\right\rbrace ^{\frac{11}{12}}.$$
The system \eqref{1.1} for the steady-state \eqref{1.7}  governing the Couette flow in a finite channel $\mathbb{T}\times \text{I}$ with non-slip boundary condition on the velocity has been extensively investigated and documented (\cite{ref14}). We can refer to this literature for further details, including the stability threshold
$$||\omega_{0}||_{H^1}\leq c_0\min\left\lbrace \nu,\mu\right\rbrace ^{\frac{1}{2}},\quad||\theta_{0}||_{H^1}+|||D_x|^{\frac{1}{6}}\theta_{0}||_{H^1}\leq c_0\min\left\lbrace \nu,\mu\right\rbrace ^{\frac{11}{12}}.$$

What lead to this stability are two main mechanisms: dissipative enhancement and inviscid damping, the latter of which is  the consequence of  mixing by the shear flow, which can refer to \cite{Bedrossian.2015,Ionescu.2019,Zhao.2020} for the inviscid damping results for Euler equation.
Due to the effect of diffusion, the  decay rate of the high frequency part is much faster than the low frequency part, and the energy is sent from the low frequency to the high frequency by the mixing, which gives dissipative enhancement.
The phenomenon of dissipative enhancement has been extensively recognized and studied in the physical literature (see \cite{ref40,ref41,ref42,ref43}). Recently, it has garnered significant attention from the mathematical community, with many researchers making notable progress. Constantin, Kiselev, Ryzhik, and Zlatos \cite{ref44} achieved one of the earliest rigorous results on the dissipative enhancement of diffusive mixing. Subsequently, numerous remarkable findings have been established. Notably, there has been intensive investigation into the stability of shear flows in relation to passive scale equations and Navier-Stokes equations through a series of outstanding papers (see \cite{ref20,ref51,ref22,ref23,ref24,ref26}).

When $\theta=0$, equation  \eqref{1.1} reduces to a two-dimensional Navier-Stokes equation. There is an extensive body of work \cite{ref15,ref16,ref17,ref18,ref19} in applied mathematics and physics dedicated to estimating transition thresholds for various flows such as Couette flow and Poiseuille flow.

Bedrossian, Germain, and Masmoudi \cite{48} once put forward a  mathematical version concerning this:
Given a norm $||\cdot||_X$, determine a $\gamma=\gamma(X)$ such that 
$$||u_0||_X<C \nu^{\gamma}\rightrightarrows stability,$$
$$||u_0||_X\gg C \nu^{\gamma}\rightrightarrows instability,$$
$\gamma$: the transition threshold in the applied literature.

The stability threshold problem in Couette flow has been extensively investigated in a series of studies  \cite{ref20,Bedrossian.2022,ref51,ref22,ref23,ref25,ref47,ref26,ref13,ref24}. In summary:

When $\Omega=\mathbb{T}\times\mathbb{R}\times\mathbb{T},$
\begin{itemize}
	\item  if the perturbation for velocity is in Gevrey class, then $\gamma=1$ \cite{ref20,Bedrossian.2022},
	\item  if the perturbation for velocity is in Sobolev class $H^{\sigma}(\sigma>\frac{9}{2})$, then $\gamma\leq\frac{3}{2}$ \cite{ref51},
	\item  if the perturbation for velocity is  in Sobolev class $H^{2}$, then $\gamma\leq 1$ \cite{ref24}. 
\end{itemize}
When $\Omega=\mathbb{T}\times[-1,1]\times\mathbb{T},$
\begin{itemize}
	\item  if the perturbation  for velocity is  in Sobolev class $H^{2}$, then $\gamma\leq 1$ \cite{ref26}.
\end{itemize}
When $\Omega=\mathbb{T}\times \mathbb{R},$
\begin{itemize}
	\item  if the perturbation for vorticity is in the Gevery class, then $\gamma=0$ \cite{ref22},
	\item  if the perturbation for velocity is in Sobolev space  $H^2$, then       $\gamma\leq\frac{1}{2}$ \cite{ref23}, 
	\item  if the perturbation for vorticity is in Sobolev space  $H_x^{log}L_y^2$, then $\gamma\leq\frac{1}{2}$ \cite{ref47},
	\item if the initial perturbation for vorticity is in the some Gevrey-$\frac{1}{s}$, then $\gamma\in [0,\frac{1}{3}]$ \cite{ref50},
	\item  if the perturbation for vorticity is in Sobolev space  $H^{\sigma}(\sigma\geq 40)$, then $\gamma\leq\frac{1}{3}$ \cite{ref13}.
\end{itemize}
When $\Omega=\mathbb{T}\times I,$
\begin{itemize}
	\item  if the perturbation for velocity is in Sobolev space  $H^2$, then       $\gamma\leq\frac{1}{2}$ \cite{ref25}.
\end{itemize}
Moreover, Coti Zelati et al. \cite{ref34}  demonstrated the transition threshold $\gamma\leq \frac{3}{4}+2\epsilon $  (for any $\epsilon>0$)  for 2D plane Poiseuille flow without boundary using the hypocoercivity method. This result has been further improved to $\gamma\leq \frac{2}{3}+$  by Del Zotto \cite{46}. Ding and Lin \cite{ref35} investigated the case with Navier-slip boundary in a finite channel $\mathbb{T}\times (-1,1),$  and established the transition threshold $\gamma\leq \frac{3}{4}$ based on the resolvent estimate, while also providing enhanced dissipation with a decay rate of  $e^{-c\sqrt{\nu}t}.$ It is worth mentioning that their results have recently been extended to $\gamma\leq\frac{2}{3}$ \cite{ref53}.

For the 3D case, significant progress has been made by several researchers. Chen, Wei, and Zhang \cite{ref36} established the linear stability of 3D pipe Poiseuille flow through resolvent estimates. Recently, Chen, Ding et al. \cite{ref52} determined the transition threshold $\gamma\leq \frac{7}{4} $ for 3D plane Poiseuille flow using resolvent estimates. Notably, Wang and Xie \cite{ref37} demonstrated the structural stability of Hagen-Poiseuille flows in a pipe under steady conditions.

In this paper, our primary focus is on investigating the Navier-slip boundary condition for velocity in order to mitigate the impact of the boundary layer effect. It should be noted that due to this specific boundary condition, we are unable to employ the time-dependent multiplier method as mentioned in \cite{ref13}. The key findings of our study are summarized below.

\begin{theorem}\label{thm1}
	Suppose  that $(u,\theta)$ be the solution the system \eqref{1.3} with the initial data $(u_{0},\theta_{0})$.  Then there exists constants $\nu_0$  and $c_0,c_1, C> 0$ independent  of $\nu$ and $\mu$ so that if 
	\begin{equation}
		||u_{0}||_{H^{\frac{7}{2}+}}\leq c_0\min\left\lbrace \mu,\nu\right\rbrace ^{\frac{2}{3}},\quad ||\theta_{0}||_{H^1}+|||D_x|^{\frac{1}{8}}\theta_{0}||_{H^1}\leq c_1\min\left\lbrace \mu,\nu\right\rbrace ^{{\frac{31}{24}}}
	\end{equation}
	for some sufficiently small $c_0,c_1,0<\min\left\lbrace \nu,\mu\right\rbrace \leq \nu_0 ,$ then the solution $(u,\theta)$ of system \eqref{1.3} is global in time and satisfies the following stability estimate:
	\begin{equation}
		\sum_{k\in\mathbb{Z}}E_k\leq Cc_0\min\left\lbrace \mu,\nu\right\rbrace ^{\frac{2}{3}},\quad\sum_{k\in\mathbb{Z}}H_k\leq Cc_1\min\left\lbrace \mu,\nu\right\rbrace ^{\frac{31}{24}},
	\end{equation}
	where the energy functional $E_k,H_k$ are defined by
	\begin{equation}
		E_k:=	\left\{
		\begin{aligned}
			&||\hat{\omega}_k||_{L^\infty L^2}+(\nu |k|)^{\frac{1}{4}}||\hat{\omega}_k||_{L^2L^2}+(\nu |k|^2)^{\frac{1}{2}}||\hat{\omega}_k||_{L^2L^2}+|k|^{\frac{1}{2}} ||\hat{u}_k||_{L^2L^2},&k\neq 0,\\
			&\|\bar{\omega}\|_{L^\infty L^2},&k=0,\\
		\end{aligned}
		\right.
	\end{equation}
	and
	\begin{equation}
		H_k:=	\left\{
		\begin{aligned}
			&|k|^{\frac{1}{8}}||\hat{\theta}_k||_{L^\infty L^2}+\mu^{\frac{1}{4}} |k|^{\frac{3}{8}}||\hat{\theta}_k||_{L^2L^2}+\mu^{\frac{1}{2}} |k|^{\frac{9}{8}}||\hat{\theta}_k||_{L^2L^2},&k\neq 0,\\
			&\|\bar{\theta}\|_{L^\infty L^2},&k=0.\\
		\end{aligned}
		\right.
	\end{equation}
	The Fourier transform of $f$ in the $x$ direction is defined by 
	$$\hat{f}(t,k,y)=\dfrac{1}{2\pi}\int_{\mathbb{T}}f(t,x,y)e^{-ikx}dx,$$ 
	and $k$ is the wave number. Let us define 
	$$\bar{f}(t;y):=\dfrac{1}{2\pi}\int_{\mathbb{T}}f(t,x,y)dx,\quad f_{\neq}=f-\bar{f},$$
	denoting  
	the zero mode and the non-zero mode.
\end{theorem}
The results have been subject to some critical observations.
\begin{remark}
	The asymptotic stability is valid for the initial perturbation that meets the given conditions
	$$||u_{0}||_{H^{\frac{7}{2}+}}\leq c_0\min\left\lbrace \mu,\nu\right\rbrace ^{\frac{2}{3}},\quad ||\theta_{0}||_{H^1}+|||D_x|^{\frac{1}{8}}\theta_{0}||_{H^1}\leq c_1\min\left\lbrace \mu,\nu\right\rbrace ^{{\frac{31}{24}}}.$$
\end{remark}
\begin{remark}
	The estimate 	$|k|^{\frac{1}{2}} ||\hat{u}_k||_{L^2L^2}$  is attributed to the inviscid damping, while estimates $(\nu |k|)^{\frac{1}{4}}||\hat{\omega}_k||_{L^2L^2}$ and $\mu^{\frac{1}{4}} k^{\frac{3}{8}}||\hat{\theta}_k||_{L^2L^2}$ are associated with the enhanced dissipation.
\end{remark}
\begin{remark}
	The stability of the Poiseuille flow, as described in equation  \eqref{1.4}, is facilitated by the augmented dissipation and inviscid damping generated by the non-self-adjoint operator $(1-y^2)\partial_x-\nu\Delta+2\partial_x\Delta^{-1},$ which represents the linear component of system equation  \eqref{1.4}. Early studies conducted by Hörmander  \cite{ref39,ref38} focused on this type of operators. It is well-known that for the heat equation $\partial_t f-\nu\Delta f=0,$ the dissipation time scale is $O(\nu^{-1})$; however, when incorporating a transport term into the heat equation $$\partial_t f+(1-y^2)\partial_x f-\nu\Delta f=0,$$  the dissipation time scale becomes $O(\nu^{-\frac{1}{2}})$, which is significantly faster than $O(\nu^{-1})$ for small $\nu.$ This intensified dissipation effect plays a crucial role in addressing the stability problem under investigation.
\end{remark}
\begin{remark}
	The inviscid damping effect also plays an exceedingly significant role in the stability problem. In the case of Poiseuille flow, Orr \cite{ref55} observed a crucial phenomenon where the velocity tends to approach zero as time progresses, despite the Euler equation being a conserved system.
\end{remark}
\begin{remark}
	The Navier-slip boundary condition is imposed on the velocity in this paper to mitigate the impact of the boundary layer effect, which serves as an alternative to the non-slip boundary condition.% More precisely, 
\end{remark}
\begin{remark}
	The special implication of Theorem \ref{1.1} lies in the nonlinear stability analysis for the 2D Navier-Stokes equation. In the case where $\theta=0$, system \eqref{1.4} simplifies to the 2D Navier-Stokes vorticity equation. The outcome of Theorem \ref{1.1} aligns with \cite{ref53}, which established the transition threshold $\gamma\leq \frac{2}{3}$ based on resolvent estimation.
\end{remark}

\begin{remark}
	More
	precisely, there is one derivative loss of the buoyancy term $\partial_x\theta$. A $\frac{5}{8}$ derivative can be absorbed by using the enhanced dissipation of the vorticity, namely $(\nu |k|)^{\frac{1}{4}}||\hat{\omega}_k||_{L^2L^2}$ and the inviscid damping of the velocity, namely  $|k|^{\frac{1}{2}} ||\hat{u}_k||_{L^2L^2}.$
	A $\frac{1}{4}$ derivative can be absorbed by using the enhanced dissipation of the temperature, namely $(\mu |k|)^{\frac{1}{4}}||\hat{\theta}_k||_{L^2L^2}.$Thus, there is only $\frac{1}{8}$
	derivative loss in the buoyancy term that needs to be controlled,
	which requires the estimate of the temperature with an additional 
	$\frac{1}{8}$ derivative, namely
	the term $|k|^{\frac{1}{8}}||\hat{\theta}_k||_{L^\infty L^2}$
	in the energy.
\end{remark}
In section 2, we analyze the linearized operator of the system. In section 3, we obtained the resolvent estimates for Orr-Sommerfeld equation. Section 4 is dedicated to establishing space-time estimates for the linearized two-dimensional Boussinesq equation through resolvent estimates that have been established in section 3. Finally, in section 5, our primary focus lies on presenting a rigorous proof of nonlinear stability.

In this section, we give some illustrations of notations. Throughout this paper, we  always
assume that $|k| \geq1$ and that $C$ a positive constant independent of $\nu,\mu,k$, which may be different on each line. We
also use $a \sim b$ for $C^{-1}b\leq a\leq Cb$
and $a \lesssim b$ for $a \leq Cb $ for some constants $ C > 0.$

\section{The linearized equation  }
The linearized Boussinesq system in the vicinity of the Poiseuille flow is expressed as follows:
\begin{equation}\label{1.9}
	\left\{
	\begin{aligned}
		&\partial_t u+(1-y^2)\partial_x u+\begin{pmatrix}
			-2y	u^2\\
			0
		\end{pmatrix}-\nu \Delta u+\nabla p=\begin{pmatrix}
			0\\
			\theta
		\end{pmatrix},\\
		&\partial_t\theta+(1-y^2)\partial_x\theta-\mu \Delta \theta=0.\\
		%	&\theta(0,x)=\theta_0(x),u(0,x)=u_0(x),
	\end{aligned}
	\right.
\end{equation}
The presence of the pressure term in system $\eqref{1.9}_1$ poses challenges in directly addressing the velocity field. By applying the curl  $(\nabla\times )$ on both sides of $\eqref{1.9}_1$, not only can we eliminate the pressure term, but also transform the vector equation into a scalar equation.
The vorticity $\omega$ and the
stream function $\varphi$	 are defined by
$$\omega=\nabla\times u=\partial_y u^1-\partial_x u^2,\quad u=\nabla^{\perp}\varphi=(\partial_y\varphi,-\partial_x \varphi),$$which satisfies
\begin{equation}\label{1.10}
	\left\{
	\begin{aligned}
		&\partial _t\omega+(1-y^2)\partial_x\omega -\nu \Delta \omega+2\partial_x\Delta^{-1}\omega=-\partial_x\theta ,\\
		&\partial_t\theta+(1-y^2)\partial_x\theta -\mu \Delta \theta=0.\\
		%	&\theta(0,x)=\theta_0(x),u(0,x)=u_0(x),\\
	\end{aligned}
	\right.
\end{equation}
Taking the Fourier transform in \eqref{1.10}%$x \in\mathbb{T}$, we get
$$\omega(t,x,y)=\sum_{k\in\mathbb{Z}}\hat{\omega}_k(t,y)e^{ikx}=\sum_{k\in\mathbb{Z}}e^{ikx}(\partial^2_y-k^2)\hat{\varphi}_k(t,y),$$
$$\theta(t,x,y)=\sum_{k\in\mathbb{Z}}\hat{\theta}_k(t,y)e^{ikx},$$
then, we have
\begin{equation}\label{1.11}
	\left\{
	\begin{aligned}
		&\partial_t\hat{\omega}_k(t,y)+\mathcal{\hat{L}}_\nu\hat{\omega}_k(t,y)+ik\hat{\theta}_k=0,\\
		&\partial_t\hat{\theta}_k(t,y)+\mathcal{\hat{H}}_\mu\hat{\theta}_k(t,y)=0,\\
	\end{aligned}
	\right.
\end{equation}
where $$\mathcal{\hat{L}}_\nu=\nu(k^2-\partial^2_y)+ik(1-y^2)+2ik(\partial^2_y-k^2)^{-1},\quad \mathcal{\hat{H}}_\mu=\mu(k^2-\partial^2_y)+ik(1-y^2).$$
The linearized equation is considered under the Navier-slip boundary condition for velocity $$	\partial_y	u^1(t,x,\pm1),\quad u^2(t,x,\pm1)=0,$$ and the Dirichlet boundary condition for temperature
$$\theta(t,x,\pm1)=0.$$
Thus, for $k\neq 0,$ we have 
$$\hat{\varphi}_k(t,\pm1)=\hat{\varphi}''_k(t,\pm1)=0,$$
which also means 
$\hat{\omega}_k(t,\pm1)=0,$ 
and 
$$\hat{\theta}_k(t,\pm1)=0.$$

The most commonly employed approach for investigating stability involves conducting eigenvalue analysis on linearized equations. In other words, we aim to derive formal solutions
$$\hat{\omega}_k(t,y)=\omega_{k}(y)e^{-ikt\lambda},\quad\hat{\varphi}_k(t,y)=\varphi_k(y)e^{-ikt\lambda},\quad \hat{\theta}_k(t,y)=\theta_{k}(y)e^{-ikt\lambda}.$$
Then $\omega_{k}(y)$ and $\theta_{k}(y)$ satisfy the following Orr-Sommerfeld equation
\begin{equation}\label{1.12}
	\left\{
	\begin{aligned}
		&\nu(k^2-\partial^2_y)\omega_{k}(y)+ik(1-y^2-\lambda)\omega_{k}(y)+2ik(\partial^2_y-k^2)^{-1}\omega_{k}(y)=-ik\theta_{k}(y),\\
		&\mu(k^2-\partial^2_y)\theta_{k}(y)+ik(1-y^2-\lambda)\theta_{k}(y)=0,\\
		&\varphi_k(t,\pm1)=\varphi''_k(t,\pm1)=0,\\
		&\theta_k(t,\pm1)=0.
	\end{aligned}
	\right.
\end{equation}
If there exists a nontrivial solution of \eqref{1.12} 
for $\lambda\in\mathbb{C}, k > 0 $ with $Im\lambda > 0$, we say that the flow is linearly unstable.

Firstly, for the linearized vorticity equation of the system of \eqref{1.12}, due to the Navier-slip boundary condition  and the nonlocal term, the key point is to estabilish the resolvent estimate for the inhomogeneous Orr-Sommerfeld equation
\begin{equation}\label{1.13}
	\left\{
	\begin{aligned}
		&\nu(k^2-\partial^2_y)\omega_{k}(y)+ik(1-y^2-\lambda)\omega_{k}(y)+2ik(\partial^2_y-k^2)^{-1}\omega_{k}(y)=F,\\
		&(\partial^2_y-k^2)\varphi_k(y)=\omega_k(y),\\
		&\varphi_k(t,\pm1)=\varphi''_k(t,\pm1)=0.\
	\end{aligned}
	\right.
\end{equation}
Fortunately, this resolvent estimate of the \eqref{1.13} is given in the \cite{ref52}. In equation \eqref{1.12}, the buoyancy term is a linear term, distinguishing it from the linearized operator in the Navier-Stokes equation. Fortunately, the linearized temperature equation is decoupled from the entire system, allowing us to still utilize the linearized operator from the Navier-Stokes equation. In this paper, we can obtain linear estimates of velocity and vorticity using the same method as \cite{ref53}. To streamline this paper, we will employ certain linear estimates from \cite{ref53} as a black box.

Secondly, for the linearized temperature equation of the system of \eqref{1.12}, 
\begin{equation}\label{4.1}
	\left\{
	\begin{aligned}
		&\mu(k^2-\partial^2_y)\theta_{k}(y)+ik(1-y^2-\lambda)\theta_{k}(y)=F,\\
		&	\theta_k(t,\pm1)=0,\\
	\end{aligned}
	\right.
\end{equation}
thanks to  the good  boundary condition  and the lack of  the  nonlocal term, we use the basic resolvent  estimates  to estabilish the space-time estimates, which can be seen
at Section \ref{sec4.2} for more details.

\section{ Resolvent estimates for Orr-Sommerfeld equation}

In this section, we study the resolvent estimates for the Orr-Sommerfeld equation around the plane Poiseuille flow as follows
\begin{equation}\label{1.15}
\left\{\begin{array}{l}
-\nu\left(\partial_y^2-|k|^2\right) w+i k\left[\left(1-y^2-\lambda\right) w+2 \varphi\right]=F, \\
\left(\partial_y^2-|k|^2\right) \varphi=w, \varphi( \pm 1)=w( \pm 1)=0,
\end{array}\right.
\end{equation}
where $k \neq 0$ and $F \in H_0^1$.
The resolvent estimates for \eqref{1.15} are stated as follows.

\begin{proposition} Let $w$ be the solution to \eqref{1.15}, then it holds that
$$
\begin{aligned}
& \nu^{\frac{3}{8}}|k|^{\frac{5}{8}}\left(\|w\|_{L^1}+|k|^{\frac{1}{2}}\|u\|_{L^2}\right) \\
& \quad+(\nu|k|)^{\frac{1}{2}}\|w\|_{L^2}+\nu^{\frac{3}{4}}|k|^{\frac{1}{4}}\left\|\left(\partial_y,|k|\right) w\right\|_{L^2} \lesssim\|F\|_{L^2}, \\
& \nu^{\frac{3}{4}}|k|^{\frac{1}{4}}\|w\|_{L^2}+\nu\left\|\left(\partial_y,|k|\right) w\right\|_{L^2}+(\nu|k|)^{\frac{1}{2}}\|u\|_{L^2} \lesssim\|F\|_{H_k^{-1}}, \\
& \quad|\nu / k|^{\frac{1}{2}}\left\|\left(\partial_y,|k|\right) w\right\|_{L^2}+|\nu / k|^{\frac{1}{4}}\|w\|_{L^2} \\
& \quad+|\nu / k|^{\frac{1}{8}}\|u\|_{L^{\infty}}+\|u\|_{L^2} \lesssim|k|^{-1}\left\|\left(\partial_y,|k|\right) F\right\|_{L^2} .
\end{aligned}
$$
Moreover, there holds that
$$
\begin{aligned}
& \nu^{\frac{1}{6}}|k|^{\frac{5}{6}}\left(|\lambda-1|^{\frac{1}{2}}+|\nu / k|^{\frac{1}{4}}\right)^{\frac{1}{3}}\|u\|_{L^2} \\
& \quad+\nu^{\frac{2}{3}}|k|^{\frac{1}{3}}\left(|\lambda-1|^{\frac{1}{2}}+|\nu / k|^{\frac{1}{4}}\right)^{\frac{1}{3}}\left\|\left(\partial_y,|k|\right) w\right\|_{L^2} \\
& \quad+\nu^{\frac{1}{3}}|k|^{\frac{2}{3}}\left(|\lambda-1|^{\frac{1}{2}}+|\nu / k|^{\frac{1}{4}}\right)^{\frac{2}{3}}\|w\|_{L^2} \lesssim\|F\|_{L^2}, \\
& \nu^{\frac{2}{3}}|k|^{\frac{1}{3}}\left(|\lambda-1|^{\frac{1}{2}}+|\nu / k|^{\frac{1}{4}}\right)^{\frac{1}{3}}\|w\|_{L^2} \lesssim\|F\|_{H_k^{-1}} .
\end{aligned}
$$

Here, $\|F\|_{H_k^{-1}}:=\inf _{\left\{g \in H_0^1:\|g\|_{H_k^1}=1\right\}}\langle f, g\rangle$ with $\|g\|_{H_k^1}=\left\|\left(\partial_y,|k|\right) g\right\|_{L^2}$.
\end{proposition}
\begin{proof}
 See \cite{ref52}.
\end{proof}
\section{Space-time estimates of the linearized Boussinesq equation}
The space-time estimates of the linearized two-dimensional Boussinesq equation are established in this section.
\subsection{Space-time estimates for the vorticity}\label{sec 4.1}\

First, we consider  the linearized system for vorticity equation  as follows:
\begin{equation}\label{1.16}
\left\{\begin{array}{l}
\partial_t \omega-\nu\left(\partial_y^2-k^2\right)\omega+i k\left(1-y^2\right) \omega+2ik \varphi=-i k f_1-\partial_y f_2-f_3-f_4, \\
\left(\partial_y^2-k^2\right) \varphi=\omega, \quad u=\left(\partial_y,-i k\right) \varphi, \\
\varphi(t, k, \pm 1)=\omega(t, k, \pm 1)=0, \\
\left.\omega\right|_{t=0}=\omega_0(k, y),
\end{array}\right.
\end{equation}
where $\left.f_3\right|_{y= \pm 1}=0$ and $k \neq 0$.
The space-time estimates for \eqref{1.16} are stated as follows.

\begin{proposition}\label{pro4.1}
	( see \cite{ref53}) Suppose the $\omega_0 \in H_0^1 \cap H^2$, then the solution $\omega$ to \eqref{1.16} enjoys that the following estimates:
$$
\begin{aligned}
& \|\omega\|_{L^{\infty} L^2}^2+\nu\left\|\left(\partial_y,|k|\right) \omega\right\|_{L^2 L^2}^2+(\nu|k|)^{\frac{1}{2}}\|\omega\|_{L^2 L^2}^2+|k|\|u\|_{L^2 L^2}^2 \\
& \lesssim\left\|\Delta_k \omega_0\right\|_{L^2}^2+\nu^{-1}\left\|\left(f_1, f_2\right)\right\|_{L^2 L^2}^2 \\
& \quad+|k|^{-1}\left\|\left(\partial_y,|k|\right) f_3\right\|_{L^2 L^2}^2+\min \left\{(\nu|k|)^{-\frac{1}{2}},\left(\nu|k|^2\right)^{-1}\right\}\left\|f_4\right\|_{L^2 L^2}^2 .
\end{aligned}
$$
Moreover, there exists $0<c \ll 1$ such that
$$
\begin{aligned}
& \left\|e^{c \nu^{\frac{1}{2}} t} \omega\right\|_{L^{\infty} L^2}^2+\nu\left\|e^{c \nu^{\frac{1}{2}} t}\left(\partial_y,|k|\right) \omega\right\|_{L^2 L^2}^2 \\
& +(\nu|k|)^{\frac{1}{2}}\left\|e^{c \nu^{\frac{1}{2}} t} \omega\right\|_{L^2 L^2}^2+|k|\left\|e^{c \nu^{\frac{1}{2}} t} u\right\|_{L^2 L^2}^2 \\
& \lesssim\left\|\Delta_k \omega_0\right\|_{L^2}^2+\nu^{-1}\left\|e^{c \nu^{\frac{1}{2}} t}\left(f_1, f_2\right)\right\|_{L^2 L^2}^2+|k|^{-1}\left\|e^{c\nu^{\frac{1}{2}}t} \left(\partial_y,|k|\right) f_3\right\|_{L^2 L^2}^2 \\
& \quad+\min \left\{(\nu|k|)^{-\frac{1}{2}},\left(\nu|k|^2\right)^{-1}\right\}\left\|e^{c \nu^{\frac{1}{2}} t} f_4\right\|_{L^2 L^2}^2 .
\end{aligned}
$$
\end{proposition}

\subsection{Space-time estimates for $\theta$}\label{sec4.2}\

In this section, we consider the space-time estimates for  temperature as follows:
\begin{equation}\label{1.17}
\left\{
\begin{array}{l}
	\left(\partial_t+\hat{\mathcal{H}}_\mu\right)\hat{\theta}=-i k g_1-\partial_y g_2-g_3,\\
	\hat{\theta}( \pm 1)=0,\left.\quad \hat{\theta} \right|_{t=0}=\hat{\theta}_0\left(k, y\right),
\end{array}
\right.
\end{equation}
where $k \neq 0$ and $\hat{\mathcal{H}}_\mu=-\mu\left(\partial_y^2-|k|^2\right)+i k\left(1-y^2\right)$.
By the standard energy estimates for $\hat{\theta}$, we can easily get the following proposition, which is important for the estimates of the high frequency of $\hat{\theta}$.

\begin{proposition}\label{pro4.2}
	Let $\hat{\theta}$ be a solution of \eqref{1.17} with $\hat{\theta}_{0} \in L^2(-1,1)$ and $g_1, g_2 \in L^2 L^2$. Then there exists a constant $C>0$ independent in $\mu, k$ so that
$$
\|\widehat{\theta}\|_{L^{\infty} L^2}^2+\mu k^2\|\widehat{\theta}\|_{L^2 L^2}^2+\mu\left\|\partial_y \widehat{\theta}\right\|_{L^2 L^2}^2 \leq C \mu^{-1}\left(\left\|g_1\right\|_{L^2 L^2}^2+\left\|g_2\right\|_{L^2 L^2}^2+k^{-2}\|g_3\|_{L^2L^2}^2\right)+\left\|\widehat{\theta}_{0}\right\|_{L^2}^2 .
$$

Proof. Taking $L^2$ inner product between \eqref{1.17} and $\widehat{\theta}$, we get
$$
\left\langle\partial_t \widehat{\theta}, \hat{\theta}\right\rangle-\mu\left\langle\left(\partial_y^2-k^2\right) \widehat{\theta}, \hat{\theta}\right\rangle+\langle i k \left( 1-y^2\right)  \widehat{\theta}, \widehat{\theta}\rangle=\left\langle-i k g_1-\partial_y g_2-g_3, \widehat{\theta}\right\rangle .
$$

By taking the real part and integration by parts in the above equality, we obtain
$$
\begin{aligned}
	\frac{1}{2} \frac{d}{d t}\|\hat{\theta}\|_{L^2}^2+\mu\left\|\partial_y \hat{\theta}\right\|_{L^2}^2+\mu k^2\|\widehat{\theta}\|_{L^2}^2 \leq & C\left\|g_1\right\|_{L^2}\|k \hat{\theta}\|_{L^2}+C\left\|g_2\right\|_{L^2}\left\|\partial_y \widehat{\theta}\right\|_{L^2} +C||g_3||_{L^2}\|\theta\|_{L^2}\\
	\leq & \frac{1}{4} \mu\|k \hat{\theta}\|_{L^2}^2+C \mu^{-1}\left\|g_1\right\|_{L^2}^2+\frac{1}{4} \mu\left\|\partial_y \hat{\theta}\right\|_{L^2}^2 \\
	& +C \mu^{-1}\left\|g_2\right\|_{L^2}^2+C(\mu k^2)^{-1} \|g_3\|_{L^2}^2+\frac{1}{4}\mu k^2\|\theta\|_{L^2}.
\end{aligned}
$$
Thus by integrating in time, we have
$$
\|\widehat{\theta}\|_{L^{\infty} L^2}^2+\mu k^2\|\widehat{\theta}\|_{L^2 L^2}^2+\mu\left\|\partial_y \widehat{\theta}\right\|_{L^2 L^2}^2 \leq C \mu^{-1}\left(\left\|g_1\right\|_{L^2 L^2}^2+\left\|g_2\right\|_{L^2 L^2}^2+k^{-2}\|g_3\|_{L^2L^2}^2\right)+\left\|\hat{\theta}_{0}\right\|_{L^2}^2 .
$$
\end{proposition}
To deal with the buoyancy term $\partial_x\theta$ in the vorticity equation, we also need to give the following estimate about $\hat{\theta}.$
First, we decompose $\hat{\theta}=\hat{\theta}_I+\hat{\theta}_{H}$, where $\hat{\theta}_I$ solves

\begin{equation}\label{1.18}
\left\{
\begin{array}{l}
	\left(\partial_t+\hat{\mathcal{H}}_\mu\right) \hat{\theta}_I=-i k g_1-\partial_y g_2-g_3 ,\\
	\hat{\theta}_I( \pm 1)=0,\left.\quad \hat{\theta}_I \right|_{t=0}=0,
\end{array}
\right.
\end{equation}
and 
\begin{equation}\label{1.19}
\left\{
\begin{array}{l}
	\left(\partial_t+\hat{\mathcal{H}}_\mu\right) 
	\hat{\theta}_
	H=0 ,\\
	\hat{\theta}_H( \pm 1)=0,\left.\quad \hat{\theta}_H \right|_{t=0}=\hat{\theta}_0\left(k, y\right).
\end{array}
\right.
\end{equation}
For the homogeneous part, $\hat{\theta}_H$, we use the $$\hat{\mathcal{H}}_\mu=\mu(k^2-\partial^2_y)+ik\left( 1-y^2\right) $$
is m-accretive operator.
A closed operator $H$ in a Hilbert space $X$ is called $m$-accretive if 
$$\text{Re}\left\langle Hf,f\right\rangle \geq0\quad\forall f\in X $$  and  the left open half-plane is contained in the resolvent set of $H$ \cite{ref27} with
$$(H + \lambda)^{-1}\in \mathcal{B}(X),\|(H + \lambda)^{-1}\|\leq (\text{Re}\lambda)^{-1} \quad~ \text{for} ~\text{Re}\lambda> 0.$$
Here $\mathcal{B}(X)$ is the set of bounded linear operators on X.
%for Re$\lambda$ > 0.
We define
$$\Psi(A)=\inf\left\lbrace \|(A-i\lambda)u\|:u\in D(A),\lambda \in\mathbb{R},||u||=1\right\rbrace .$$
The following Gearhart-Pr\"{u}ss type lemma comes form \cite{ref29}.
\begin{lemma}\label{lem4.3}
	let $A$ is $m$-accretive in a Hilbert space $X$, then for any $t>0$
	$$||e^{-tA}||\leq e^{-t\Psi\left( A\right) +\frac{\pi}{2}}.$$
\end{lemma}
Recalling that 
$$\hat{\mathcal{H}}_\mu=\mu(k^2-\partial^2_y)+ik\left(1-y^2\right) $$is $m$-accretive,
in follow Lemma \ref{lem4.5} that $$\Psi(\hat{{\mathcal{H}}_\mu})\geq (\mu |k|)^{\frac{1}{2}}+\mu k^2,$$
the lemma \ref{lem4.3} give the enhanced dissipation.
\begin{lemma}\label{lem4.4}
Let $\widehat{\theta}_{0} \in L^2(-1,1)$. Then for any $k \in \mathbb{Z}$, there exist constants $C>0$ independent of $\mu, k$ such that
$$
\left\|\widehat{\theta}_H\right\|_{L^\infty L^2}+\mu\|(\partial_y,k)\hat{\theta}_H\|_{L^2L^2}^2 +\left(\mu |k|\right)^{\frac{1}{2}}\left\|\widehat{\theta}_H\right\|_{L^2 L^2}^2\leq C \left\|\widehat{\theta}_{0}\right\|_{L^2}^2.
$$
\end{lemma}
\begin{proof}

By the Lemma \ref{lem4.3}, we can get 
%	$$	||\hat{\omega}_H||^2_{L^{\infty}L^2}+(\nu k^2)^{\frac{1}{3}}||\hat{\omega}_H||^2_{L^2L^2}\leq C||\hat{\omega}_{0}(k,y)||_{L^2}^2.$$
$$||\hat{\theta}_H||_{L^{\infty}L^2}\leq Ce^{-c(\nu k^2)^{\frac{1}{3}}t-\nu k^2t}||\hat{\theta}_{0}(k,y)||_{L^2}$$
The second inequality is a direct consequence of the first one.
And based the basic energy estimate, we can obtain $$\mu\|(\partial_y,k)\hat{\theta}_H\|_{L^2L^2}^2 \leq C\left\|\widehat{\theta}_{0}\right\|_{L^2}^2.$$
\end{proof}
For the inhomogeneous part, considering the system
\begin{equation}\label{1.20}
	-\mu\left(\partial_y^2-k^2\right) \widehat{\Theta}+i k(1-y^2-\lambda) \widehat{\Theta}=F, \widehat{\Theta}( \pm 1)=0,
\end{equation}
we have the following sharp resolvent estimates for the linearized operator, which is very important for the space-time estimates of $\widehat{\theta}_I$.

\begin{lemma}\label{lem4.5}
	(Proposition 2.6 and Proposition 2.7 in \cite{ref52}) Let $\widehat{\Theta} \in H^2(-1,1)$ be a solution of \eqref{1.20} with $\lambda \in \mathbb{R}$. Then it holds for $F \in L^2(-1,1)$,
$$
\mu^{\frac{3}{4}}|k|^{\frac{1}{4}}\left\|\partial_y \widehat{\Theta}\right\|_{L^2}+\left(\mu |k|\right)^{\frac{1}{2}}\|\widehat{\Theta}\|_{L^2}+\mu^{\frac{3}{4}}|k|^{\frac{1}{4}}\left\|\left(\partial_y,|k|\right) \widehat{\Theta}\right\|_{L^2} \leq C\|F\|_{L^2},
$$
and for $F \in H^{-1}_k(-1,1)$,
$$
\mu\left\|\partial_y \widehat{\Theta}\right\|_{L^2}+\mu^{\frac{3}{4}}|k|^{\frac{1}{4}}\|\widehat{\Theta}\|_{L^2} \leq C\|F\|_{H_k^{-1}}.
$$
\end{lemma}
\begin{lemma}\label{lem4.6}
Let $\hat{\theta}_I$ be a solution for the inhomogeneous part, we have the following estimate
\[
\begin{aligned}
	&||\hat{\theta}_I||^2_{L^\infty L^2}+(\mu |k|)^{\frac{1}{2}}||\hat{\theta}_I||^2_{L^2L^2}+\mu\|(\partial_y,k)\hat{\theta}_I\|_{L^2L^2}^2\\
	& \lesssim \mu^{-1}\left\|\left(g_1, g_2\right)\right\|_{L^2 L^2}^2+\min \left\{(\mu|k|)^{-1 / 2},\left(\mu|k|^2\right)^{-1}\right\}\left\|g_3\right\|_{L^2 L^2}^2 .
\end{aligned}
\]
\end{lemma}
\begin{proof}
Taking the  Fourier transform in $t$ 
$$\widehat{\Theta}(\lambda,k,y)=\int_{0}^{+\infty}\hat{\theta}_I(t,k,y)e^{-it\lambda}dt,\quad G_j(\lambda,k,y)=\int_{0}^{+\infty}{g}_j(t,k,y)e^{-it\lambda}dt,j=1,2,3$$
%	and $$\Phi(\lambda,k,y)=\int_{0}^{+\infty}\tilde{\Phi}_I(t,k,y)e^{-it\lambda}dt$$
then we have 
$$\mu(k^2-\partial^2_y){\widehat{\Theta}}(\lambda,k,y)+ik(1-y^2+\frac{\lambda}{k}){\widehat{\Theta}}(\lambda,k,y)=-ikG_1-\partial_y G_2+G_3,{\widehat{\Theta}}(\lambda,k,\pm1)=0.$$
Using Plancherel’s theorem, we know that
$$\int_{\mathbb{R}}||{\widehat{\Theta}}(\lambda)||^2_{L^2}d\lambda=\int_{0}^{\infty}||\hat{\theta}_I(t)||^2_{L^2}dt,$$
$$\int_{0}^{\infty}||g_j(t)||^2_{L^2}dt=\int_{\mathbb{R}}||G_j(\lambda)||^2_{L^2}d\lambda\quad j=1,2,3.$$
%	$$\int_{0}^{\infty}||\hat{u}_I(t)||^2_{L^2}dt=\int_{\mathbb{R}}||(k,\partial_y)(\partial^2_y-k^2)\omega(\lambda)||^2_{L^2}d\lambda.$$
We first decompose $\widehat{\Theta}(\lambda,k,y)=\widehat{\Theta}^{(1)}+\widehat{\Theta}^{(2)},$
where $\widehat{\Theta}^{(1)}$ and $\widehat{\Theta}^{(2)}$ solve
\begin{equation}
	\left\{
\begin{aligned}
&\mu(k^2-\partial^2_y){\widehat{\Theta}}^{(1)}(\lambda,k,y)+ik(1-y^2+\frac{\lambda}{k}){\widehat{\Theta}}^{(1)}(\lambda,k,y)=-ikG_1-\partial_y G_2, \\
&{\widehat{\Theta}}^{(1)}(\lambda,k,\pm1)=0,
\end{aligned}
\right.
\end{equation}
and 
\begin{equation}
	\left\{
	\begin{aligned}
&\mu(k^2-\partial^2_y){\widehat{\Theta}}^{(2)}(\lambda,k,y)+ik(1-y^2+\frac{\lambda}{k}){\widehat{\Theta}}^{(2)}(\lambda,k,y)= G_3,\\
&{\widehat{\Theta}}^{(2)}(\lambda,k,\pm1)=0.
\end{aligned}
\right.
\end{equation}
By Lemma \ref{lem4.5},
we have $$\mu^{\frac{3}{4}}|k|^{\frac{1}{4}}||\widehat{\Theta}^{(1)}||_{L^2}
\lesssim||G_1,G_2||_{L^2},$$
and 
$$	\mu^{\frac{3}{4}}|k|^{\frac{5}{4}}||\widehat{\Theta}^{(2)}||_{L^2}\lesssim||G_3||_{L^2},$$
or
$$	\mu^{\frac{1}{2}}|k|^{\frac{1}{2}}||\widehat{\Theta}^{(2)}||_{L^2}\lesssim||G_3||_{L^2},$$
so we can get
$$(\mu|k|)^{\frac{1}{2}}||\widehat{\Theta}^{(2)}||_{L^2}^2\leq\min \left\{(\mu|k|)^{-\frac{1}{2}},\left(\mu|k|^2\right)^{-1}\right\}||G_3||_{L^2}^2.$$
Then, by Plancherel’s theorem, we have
\begin{equation}
\begin{aligned}
&	(\mu|k|)^{\frac{1}{2}}||\hat{\theta}_I(t,k,y)||^2_{L^2L^2}
	=(\mu|k|)^{\frac{1}{2}}\big\| ||\widehat{\Theta}(\lambda,k,y)||_{L^2}\big\|^2_{L^2}\\
	&\leq(\mu|k|)^{\frac{1}{2}}\big\| ||\widehat{\Theta}^{1}(\lambda,k,y)||_{L^2}+||\widehat{\Theta}^{2}(\lambda,k,y)||_{L^2}\big\|^2_{L^2}\\
	%	&+\nu^{\frac{1}{4}}|k|^{\frac{7}{4}}\big\|||(k,\partial_y)(\partial^2_y-k^2)^{-1}\omega^{(1)}(\lambda)||_{L^2}\big\|^2_{L^2}+\nu^{\frac{1}{4}}|k|^{\frac{7}{4}}\big\|||(k,\partial_y)(\partial^2_y-k^2)^{-1}\omega^{(2)}(\lambda)||_{L^2}\big\|^2_{L^2}\\
	&\leq (\mu|k|)^{\frac{1}{2}}\big\|\mu^{-\frac{3}{4}}|k|^{-\frac{1}{4}} ||G_1,G_2||_{L^2}\big\|^2_{L^2}+\min \left\{(\mu|k|)^{-\frac{1}{2}},\left(\mu|k|^2\right)^{-1}\right\}\big\|||G_3(\lambda)||_{L^2}\big\|^2_{L^2}\\
	%	&+\nu^{\frac{1}{4}}|k|^{\frac{7}{4}}\big\|\nu^{-\frac{3}{8}}|k|^{-\frac{9}{8}}||kF^1(\lambda)||_{L^2}\big\|^2_{L^2}+\nu^{\frac{1}{4}}|k|^{\frac{7}{4}}\big\|\nu^{-\frac{5}{8}}|k|^{-\frac{7}{8}}||F^{(2)}(\lambda)||_{L^2}\big\|^2_{L^2}\\
	&\leq\mu^{-1}\big\|||G_1(\lambda),G_2(\lambda)||_{L^2}\big\|^2_{L^2}+\min \left\{(\mu|k|)^{-\frac{1}{2}},\left(\mu|k|^2\right)^{-1}\right\}\big\|||G_3(\lambda)||_{L^2}\big\|^2_{L^2}\\
	&\leq\mu^{-1}||g_1,g_2||^2_{L^2L^2}+\min \left\{(\mu|k|)^{-\frac{1}{2}},\left(\mu|k|^2\right)^{-1}\right\}||g_3||^2_{L^2L^2}.\\
\end{aligned}
\end{equation}
Based on the basic energy estimates, then we can obtain
	$$
	\begin{aligned}
		& \mu\left\|\left(\partial_y,|k|\right) \hat{\theta}_I\right\|_{L^2 L^2}^2 \\
		& \lesssim \mu^{-1}\left\|\left(g_1, g_2\right)\right\|_{L^2 L^2}^2+\left(\mu|k|^2\right)^{-1}\left\|g_3\right\|_{L^2 L^2}^2
	\end{aligned}
	$$
	and
	$$
	\begin{aligned}
		& \mu|k|^2\left\|\hat{ \theta}_I\right\|_{L^2 L^2}^2 \\
		& \lesssim\mu^{-1}\left\|\left(g_1, g_2\right)\right\|_{L^2 L^2}^2+\left\|g_3\right\|_{L^2 L^2}\|\hat{ \theta}_I\|_{L^2 L^2} \\
		& \lesssim \nu^{-1}\left\|\left(g_1, g_2\right)\right\|_{L^2 L^2}^2+(\mu|k|)^{-\frac{1}{2}}\left\|g_3\right\|_{L^2 L^2}^2 .
	\end{aligned}
	$$
	we can obtain 
	$$
	\begin{gathered}
		(\mu|k|)^{\frac{1}{2}}\left\|\hat{\theta}_I\right\|_{L^2 L^2}^2+ \mu\|(\partial_y,k)\hat{\theta}_I\|_{L^2L^2}^2\\
		\lesssim \mu^{-1}\left\|\left(g_1, g_2\right)\right\|_{L^2 L^2}^2+\min \left\{(\mu|k|)^{-\frac{1}{2}},\left(\mu|k|^2\right)^{-1}\right\}\left\|g_3\right\|_{L^2 L^2}^2,
	\end{gathered}
	$$
	It remains to bound  $\left\|\hat{\theta_I}\right\|_{L^{\infty} L^2}$.
	Indeed, we have
	$$
	\begin{aligned}
		& \frac{1}{2} \frac{\mathrm{d}}{\mathrm{d} t}\left\|\hat{\theta}_I\right\|_{L^2}^2+\mu\left\|\left(\partial_y,|k|\right) \hat{\theta}_I\right\|_{L^2}^2=\operatorname{Re}\left\langle-i k g_1-\partial_y g_2-g_3,\hat{\theta}_I\right\rangle \\
		& \leq\left\|\left(g_1, g_2\right)\right\|_{L^2}\left\|\left(\partial_y,|k|\right) \hat{\theta}_I\right\|_{L^2}+\left\|g_3\right\|_{L^2}\left\|\hat{\theta}_I\right\|_{L^2} \\
		& \leq \frac{1}{2}\mu \left\|\left(\partial_y,|k|\right) \hat{\theta}_I\right\|_{L^2}^2+\mu^{-1}\left\|\left(g_1, g_2\right)\right\|_{L^2}^2+\left(\mu|k|^2\right)^{-1}\left\|g_3\right\|_{L^2}^2 \\
		&
	\end{aligned}
	$$
	or
	$$
	\begin{aligned}
			& \frac{1}{2} \frac{\mathrm{d}}{\mathrm{d} t}\left\|\hat{\theta}_I\right\|_{L^2}^2+\mu\left\|\left(\partial_y,|k|\right) \hat{\theta}_I\right\|_{L^2}^2=\operatorname{Re}\left\langle-i k g_1-\partial_y g_2-g_3,\hat{\theta}_I\right\rangle \\
		& \leq\left\|\left(g_1, g_2\right)\right\|_{L^2}\left\|\left(\partial_y,|k|\right) \hat{\theta}_I\right\|_{L^2}+\left\|g_3\right\|_{L^2}\left\|\hat{\theta}_I\right\|_{L^2} \\
		\leq &\frac{1}{2}\mu \left\|\left(\partial_y,|k|\right) \hat{\theta}_I\right\|_{L^2}^2+\mu^{-1}\left\|\left(g_1, g_2\right)\right\|_{L^2}^2\\
		& +\frac{1}{2}\left((\mu|k|)^{-1 / 2}\left\|g_3\right\|_{L^2}^2+(\mu|k|)^{\frac{1}{2}}\left\|\hat{\theta}_I\right\|_{L^2}^2\right),
	\end{aligned}
	$$
	which gives that
	$$
	\begin{aligned}
		&\frac{1}{2} \frac{\mathrm{d}}{\mathrm{d} t}\left\|\hat{\theta}_I\right\|_{L^2}^2 +\mu\left\|\left(\partial_y,|k|\right) \hat{\theta}_I\right\|_{L^2}^2\\
		& \lesssim \mu^{-1}\left\|\left(g_1, g_2\right)\right\|_{L^2}^2+\min \left\{(\mu|k|)^{-1 / 2},\left(\mu|k|^2\right)^{-1}\right\}\left\|g_3\right\|_{L^2}^2+(\mu|k|)^{\frac{1}{2}}\left\|\hat{ \theta}_I\right\|_{L^2}^2 .
	\end{aligned}
	$$
	Therefore, we have
	$$
	\begin{aligned}
		& \quad\left\|\hat{\theta}_I\right\|_{L^{\infty} L^2}^2+\mu\left\|\left(\partial_y,|k|\right) \hat{\theta}_I\right\|_{L^2}^2 \\
		& \lesssim \mu^{-1}\left\|\left(g_1, g_2\right)\right\|_{L^2 L^2}^2
		 +\min \left\{(\mu|k|)^{-1 / 2},\left(\mu|k|^2\right)^{-1}\right\}\left\|g_3\right\|_{L^2 L^2}^2+(\mu|k|)^{\frac{1}{2}}\left\|\hat{\theta}_I\right\|_{L^2 L^2}^2 \\
		& \lesssim \mu^{-1}\left\|\left(g_1, g_2\right)\right\|_{L^2 L^2}^2+\min \left\{(\mu|k|)^{-1 / 2},\left(\mu|k|^2\right)^{-1}\right\}\left\|g_3\right\|_{L^2 L^2}^2 .
	\end{aligned}
	$$
	
	The proof is completed.

\end{proof}
Thus, combining Lemma \ref{lem4.4} and Lemma \ref{lem4.6}, we immediately obtain the following space-time estimates of $\widehat{\theta}$.

\begin{proposition}\label{pro4.6}
	Let $\widehat{\theta}$ be a solution of \eqref{1.17} with $\widehat{\theta}_{0} \in L^2(-1,1)$ and $g_1, g_2 \in L^2 L^2$. Then there exists a constant $C>0$ independent in $\mu, k$ such that
\[
\begin{aligned}
	&||\hat{\theta}_I||^2_{L^\infty L^2}+(\mu |k|)^{\frac{1}{2}}||\hat{\theta}||^2_{L^2L^2}+\mu\|(\partial_y,k)\hat{\theta}\|_{L^2L^2}^2\\
	& \lesssim \left\|\widehat{\theta}_{0}\right\|_{L^2}^2+\mu^{-1}\left\|\left(g_1, g_2\right)\right\|_{L^2 L^2}^2+\min \left\{(\mu|k|)^{-1 / 2},\left(\mu|k|^2\right)^{-1}\right\}\left\|g_3\right\|_{L^2 L^2}^2 .
\end{aligned}
\]
\end{proposition}
\section{Nonlinear stability}
In this section, we establish the proof for Theorem \ref{thm1}. Due to the presence of buoyancy $\partial_x \theta$ in the vorticity equation, it becomes necessary to estimate $\left\||D_x|^\frac{1}{8}\theta(t)\right\|_{L^2}$  in order to effectively control the buoyancy term. For the two-dimensional Boussinesq system, global well-posedness is a well-known result when considering data

 $$u_{0}\in H^{\frac{7}{2}+}(\Omega),\quad\theta_{0}\in H^1(\Omega),\quad|D_x|^\frac{1}{8}\theta_0\in H^1(\Omega).$$ 
  The primary focus of Theorem \ref{thm1} lies in addressing the stability problem:
\begin{equation}\label{5.1}
		\sum_{k\in\mathbb{Z}}E_k\leq Cc_0\min\left\lbrace \mu,\nu\right\rbrace ^{\frac{2}{3}},\quad\sum_{k\in\mathbb{Z}}H_k\leq Cc_1\min\left\lbrace \mu,\nu\right\rbrace ^{\frac{31}{24}}.
\end{equation}
Here
\[
E_k:=	\left\{
\begin{aligned}
	&||\hat{\omega}_k||_{L^\infty L^2}+(\nu |k|)^{\frac{1}{4}}||\hat{\omega}_k||_{L^2L^2}+(\nu |k|^2)^{\frac{1}{2}}||\hat{\omega}_k||_{L^2L^2}+|k|^{\frac{1}{2}} ||\hat{u}_k||_{L^2L^2},&k\neq 0,\\
	&\|\bar{\omega}\|_{L^\infty L^2},&k=0,\\
\end{aligned}
\right.
\]
and
\[
H_k:=	\left\{
\begin{aligned}
	&|k|^{\frac{1}{8}}||\hat{\theta}_k||_{L^\infty L^2}+\mu^{\frac{1}{4}} |k|^{\frac{3}{8}}||\hat{\theta}_k||_{L^2L^2}+\mu^{\frac{1}{2}} |k|^{\frac{9}{8}}||\hat{\theta}_k||_{L^2L^2},&k\neq 0,\\
	&\|\bar{\theta}\|_{L^\infty L^2},&k=0.\\
\end{aligned}
\right.
\]
By getting the following Proposition \ref{pro5.1} and then going through  the bootstrap arguements, we can complete the estimates \eqref{5.1}.
\begin{proposition}\label{pro5.1}
	For $k\neq 0,$	it holds that, 
	\begin{equation}\label{5.2}
		E_k\leq||\Delta_k\hat{\omega}_{0}(k,y)||_{L^2}+C\nu^{-\frac{3}{8}}\mu^{-\frac{1}{4}}H_k+C\nu^{-\frac{2}{3}}\sum_{l\in\mathbb{Z}}E_lE_{k-l} ,
	\end{equation}
	and, for $k=0,$
	\begin{equation}\label{5.3}
		E_0\leq C\nu^{-\frac{1}{2}}\sum_{l\in\mathbb{Z}\setminus\{0\}}E_lE_{-l}+||\bar{\omega}_{0}||_{L^2}.
	\end{equation}
	For $H_0,$ it holds that:
	\begin{equation}\label{5.4}
		H_0\lesssim \mu^{-\frac{1}{2}}\sum_{l\in\mathbb{Z}\setminus\{0\}}|l|^{-\frac{1}{8}}E_lH_{-l}+||\bar{\theta}_{0}||_{L^2}.
	\end{equation}
	For $k\neq 0,$ we can get the following result:
	
	\item[1.]$\mu k^2\leq 1,$
	\begin{equation}\label{5.5}
	\begin{aligned}
	&	H_k\lesssim |k|^{\frac{1}{8}}||\hat{\theta}_{0}(k,y)||_{L^2}+\sum_{l\in\mathbb{Z}} \mu^{-\frac{2}{3}}E_lH_{k-l}+\nu^{-\frac{1}{8}}\mu^{-\frac{7}{16}}\sum_{l\in\mathbb{Z}\setminus\left\lbrace 0,k\right\rbrace ,|k-l|\leq\frac{|k|}{2}} E_lH_{k-l};
\end{aligned}
	\end{equation}
	\item[2.]$\mu k^2\geq 1,$
	\begin{equation}\label{5.6}	
		H_k\lesssim|k|^{\frac{1}{8}}||\hat{\theta}_{0}(k,y)||_{L^2}+C\mu^{-\frac{5}{8}}\sum_{l\in\mathbb{Z}}E_lH_{k-l}+\mu^{-\frac{3}{8}}\nu^{-\frac{1}{8}} E_kH_0.\\
	\end{equation}
\end{proposition}
\begin{proof}
	Proof of \eqref {5.2}.
	The nonlinear problem for vorticity should be recalled, which satisfies
	\[
	\left\{
	\begin{aligned}
		&\partial_t\hat{\omega}_k-\nu(\partial^2_y-k^2)\hat{\omega}_k+ik(1-y^2)\hat{\omega}_k+2ik\hat{\varphi}_k(t,y)=-ik\hat{\theta}_k(t,y)-ikf^1_k-\partial_y f_k^2,\\
		&(\partial^2_y-k^2)\hat{\varphi}_k=\hat{\omega}_k,\\
		&\hat{\varphi}_k(\pm1)=\hat{\varphi}_k''(\pm 1)=0,\\
		&\hat{\omega}_k|_{t=0}=\hat{\omega}_{0}(k,y).
	\end{aligned}
	\right.
	\]
	Set $E_0=||\bar{\omega}||_{L^\infty L^2}$ and $H_0=||\bar{\theta}||_{L^\infty L^2},$ and for $k\neq 0,$
	$$E_k=||\hat{\omega}_k||_{L^\infty L^2}+(\nu |k|)^{\frac{1}{4}}||\hat{\omega}_k||_{L^2L^2}+(\nu |k|^2)^{\frac{1}{2}}||\hat{\omega}_k||_{L^2L^2}+|k|^{\frac{1}{2}} ||\hat{u}_k||_{L^2L^2},$$
	and $$H_k=|k|^{\frac{1}{8}}||\hat{\theta}_k||_{L^\infty L^2}+\mu^{\frac{1}{4}} |k|^{\frac{3}{8}}||\hat{\theta}_k||_{L^2L^2}+\mu^{\frac{1}{2}} |k|^{\frac{9}{8}}||\hat{\theta}_k||_{L^2L^2}$$
	Denoting $$\hat{\omega}_k(t,y)=\dfrac{1}{2\pi}\int_{\mathbb{T}}\omega(t,x,y)e^{-ikx}dx,$$
	and $$f^1_k(t,y)=\sum_{l\in\mathbb{Z}}\hat{u}^1_l(t,y)\hat{\omega}_{k-l}(t,y),f^2_k(t,y)=\sum_{l\in\mathbb{Z}}\hat{u}^2_l(t,y)\hat{\omega}_{k-l}(t,y),$$
	we have 
	$$\partial_t\hat{\omega}_k(t,y)+\nu(k^2-\partial^2_y)\hat{\omega}_k(t,y)+ik(1-y^2)\hat{\omega}_k(t,y)+2ik\hat{\varphi}_k=-ik\hat{\theta}_k(t,y)-ikf^1_k(t,y)-\partial_y f^2_k(t,y).$$
	The statement in Proposition \ref{pro4.1} implies that
\begin{equation}\label{5.7}
	{E}_k \lesssim\left\|\Delta_k \hat{\omega}_0\left( k,y\right)  \right\|_{L^2}+\nu^{-\frac{3}{8}}|k|^{\frac{3}{8}}\left\|\hat{\theta}_k\right\|_{L^2 L^2}+\nu^{-\frac{3}{8}}|k|^{\frac{3}{8}}\left\|f^1_k\right\|_{L^2 L^2}+\nu^{-\frac{1}{2}}\left\|f^2_k\right\|_{L^2 L^2},
\end{equation}
	where we have used the fact that $$|k| \min \left\{(\nu|k|)^{-\frac{1}{4}},\left(\nu|k|^2\right)^{-\frac{1}{2}}\right\} \leq \nu^{-\frac{3}{8}}|k|^{\frac{3}{8}}.$$ For $k \neq 0$, it holds that
\begin{equation}\label{5.8}
	\begin{aligned}
		\left\|f^2_k\right\|_{L^2 L^2} & \leq \sum_{l \in \mathbb{Z}}\left\|\hat{u}^2_l\hat{\omega}_{k-l}\right\|_{L^2 L^2} \\
		& \leq \sum_{l \in \mathbb{Z}}\left\|\hat{u}^2_l\right\|_{L^2 L^{\infty}}\left\| \hat{\omega}_{k-l}\right\|_{L^{\infty} L^2} \\
		& \leq \sum_{l \in \mathbb{Z}}\left\|\hat{u}^2_l\right\|_{L^2 L^2}^{\frac{1}{2}}\left\| \partial_y\hat{u}^2_l\right\|_{L^2 L^2}^{\frac{1}{2}}{E}_{k-l} \\
		& \leq \sum_{l \in \mathbb{Z}}\left(|l|^{-\frac{1}{2}}{E}_l\right)^{\frac{1}{2}}\left(|l|\left\|\hat{u}^1_l\right\|_{L^2 L^2}\right)^{\frac{1}{2}} {E}_{k-l} \\
		& \leq \sum_{l \in \mathbb{Z}}\left(|l|^{-\frac{1}{2}} {E}_l\right)^{\frac{1}{2}}\left(|l|^{-\frac{1}{2}}|l| {E}_l\right)^{\frac{1}{2}}{E}_{k-l} \lesssim \sum_{l \in \mathbb{Z}}{E}_l{E}_{k-l}
	\end{aligned}
\end{equation}
	and
\begin{equation}\label{5.9}
	\begin{aligned}
		\left\|f_k^1\right\|_{L^2 L^2} \leq & \left\|\bar{u}^1\right\|_{L^{\infty} L^{\infty}}\left\| \hat{\omega}_k\right\|_{L^2 L^2}+\left\|\hat{u}^1_k\right\|_{L^2 L^{\infty}}\|\bar{\omega}\|_{L^{\infty} L^2} \\
		& +\sum_{l \in \mathbb{Z} \backslash\{0, k\}}\left\|\hat{u}^1_l\right\|_{L^{\infty} L^{\infty}}\left\| \hat{\omega}_{k-l}\right\|_{L^2 L^2} .
	\end{aligned}
\end{equation}
		It remains to estimate the terms on the right-hand side of \eqref{5.9}.
	According to the Gagliardo-Nirenberg inequality in a bounded domain, we obtain
	\begin{equation}
	\begin{aligned}
		||\hat{u}^1_k||_{L^\infty L^\infty}&\leq C||\hat{u}^1_k||^{\frac{1}{2}}_{L^\infty L^2}||\partial_y\hat{u}^1_k||^{\frac{1}{2}}_{L^\infty L^2}+\|\hat{u}^1_k\|_{L^\infty L^2}\\&
		\leq C||\hat{u}^1_k||^{\frac{1}{2}}_{L^\infty L^2}||\partial_y\hat{u}^1_k||^{\frac{1}{2}}_{L^\infty L^2}+\|\hat{u}^1_k\|^{\frac{1}{2}}_{L^\infty L^2}\|\hat{u}^1_k\|^{\frac{1}{2}}_{L^\infty L^\infty}\\
		&\lesssim ||\hat{u}^1_k||^{\frac{1}{2}}_{L^\infty L^2}||\hat{\omega}_k||^{\frac{1}{2}}_{L^\infty L^2},%+C||\hat{u}^1_k||^{\frac{1}{2}}_{L^\infty L^2}||\hat{\omega}_k||^{\frac{1}{2}}_{L^\infty L^2}\\
		%&\leq C\nu^{-\frac{1}{12}}|k|^{-\frac{2}{3}}E_k
	\end{aligned}
\end{equation}
	%$$||\hat{u}^1_k||_{L^\infty L^\infty}\leq C||\hat{u}^1_k||^{\frac{1}{2}}_{L^\infty L^2}||\partial_y\hat{u}^1_k||^{\frac{1}{2}}_{L^\infty L^2},$$
	from $$||\hat{u}_k||^2_{L^2}=||\partial_y\varphi||^2_{L^2}+|k|^2||\varphi||^2_{L^2},$$
	and  $$||\hat{\omega}_k||^2_{L^2}=||\partial_y^2\varphi||^2_{L^2}+|k|^4||\varphi||^2_{L^2}+2|k|^2||\partial_y\varphi||^2_{L^2},$$
	so, we can easy obtain
	$$|k|^2||\hat{u}_k||^2_{L^2}\lesssim||\hat{\omega}_k||^2_{L^2}.$$
	Form which, for $k\neq0,$ we can obtain that 
	\begin{equation}
		||\hat{u}^1_k||_{L^\infty L^\infty}\lesssim|k|^{-\frac{1}{2}}||\hat{\omega}_{k}||_{L^\infty L^2}.
	\end{equation}
	For $l \neq 0, k$, it holds that $|k| \lesssim|l||k-l|$, then
	\begin{equation}
	\begin{aligned}
		& \sum_{l \in \mathbb{Z} \backslash\{0, k\}}\left\|\hat{u}^1_l\right\|_{L^{\infty} L^{\infty}}\left\|\hat{\omega}_{k-l}\right\|_{L^2 L^2} \\
		& \lesssim \sum_{l \in \mathbb{Z} \backslash\{0, k\}}\left\| \hat{u}^1_l\right\|_{L^{\infty} L^2}^{\frac{1}{2}}\left\| \partial_y\hat{u}_l^1\right\|_{L^{\infty} L^2}^{\frac{1}{2}} \\
		& \times\left\| \hat{\omega}_{k-l}\right\|_{L^2 L^2}^{\frac{5}{6}}\left\|\hat{\omega}_{k-1}\right\|_{L^2 L^2}^{\frac{1}{6}} \\
		& \lesssim \sum_{l \in \mathbb{Z} \backslash\{0, k\}}|l|^{-\frac{1}{2}}{E}_l(\nu|k-l|)^{-\frac{1}{4} \cdot \frac{5}{6}}\left(\nu^{\frac{1}{2}}|k-l|\right)^{-\frac{1}{6}}{E}_{k-l} \\
		& \lesssim \nu^{-\frac{7}{24}}|k|^{-\frac{3}{8}} \sum_{l \in \mathbb{Z} \backslash\{0, k\}}|l|^{-\frac{1}{8}} {E}_l {E}_{k-1},\\
		&
	\end{aligned}
\end{equation}
	and
	\begin{equation}
	\begin{aligned}
		& \left\|\bar{u}^1\right\|_{L^{\infty} L^{\infty}}\left\| \hat{\omega}_k\right\|_{L^2 L^2}+\left\|\hat{u}_k^1\right\|_{L^2 L^{\infty}}\|\bar{\omega}\|_{L^{\infty} L^2}\\
		& \lesssim\|\bar{\omega}\|_{L^{\infty} L^2}(\nu|k|)^{-\frac{1}{4} \cdot \frac{5}{6}}\left(\nu^{\frac{1}{2}}|k|\right)^{-\frac{1}{6}} {E}_k \\
		& +\left\|\hat{u}^1_k\right\|_{L^2 L^2}^{\frac{1}{2}}\left\|\hat{\omega}_k\right\|_{L^2 L^2}^{\frac{1}{2}}\|\bar{\omega}\|_{L^{\infty} L^2} \\
		& \lesssim \nu^{-\frac{7}{24}}|k|^{-\frac{3}{8}} {E}_0 {E}_k . \\
		&
	\end{aligned}
\end{equation}
Thus, we have
	\begin{equation}\label{5.14}
	\left\|f_k^1\right\|_{L^2 L^2} \lesssim \nu^{-\frac{7}{24}}|k|^{-\frac{3}{8}}  \sum_{l \in Z}{E}_l {E}_{k-l}.
\end{equation}
	
	Combining \eqref{5.7}, \eqref{5.8} and \eqref{5.14}, we have
	$$
	\begin{aligned}
	{E}_k & \lesssim\left\|\Delta_k\hat{ \omega}_0\left( k, y\right) \right\|_{L^2}+\nu^{-\frac{3}{8}}|k|^{\frac{3}{8}}\left\|\hat{\theta}_k\right\|_{L^2 L^2}+\nu^{-\frac{3}{8}}|k|^{\frac{3}{8}}\left\|f_k^1\right\|_{L^2 L^2}+\nu^{-\frac{1}{2}}\left\|f^2_k\right\|_{L^2 L^2} \\
		& \lesssim\left\|\Delta_k\hat{ \omega}_0\left( k, y\right) \right\|_{L^2}+\nu^{-\frac{3}{8}}\mu^{-\frac{1}{4}}H_k+\nu^{-\frac{2}{3}} \sum_{l \in \mathbb{Z}} {E}_l {E}_{k-l}.
	\end{aligned}
$$
	Proof of \eqref{5.3}.
	The divergence-free condition of velocity $u$,  allows us to obtain $\bar{u}^2(t,y)=0.$
	According to $$\sum_{l\in\mathbb{Z}}\hat{u}_l^1(t,y)i(-l)\hat{u}^1_{-l}(t,y)=0,$$
	we have 
	\begin{equation}\label{5.15}
		\begin{aligned}
			&\partial_t\bar{u}^1(t,y)-\nu\partial^2_y\bar{u}^1(t,y)=-\sum_{l\in\mathbb{Z},l\neq 0}\hat{u}^2_l(t,y)\partial_y\hat{u}^1_{-l}(t,y)\\
			&=-\sum_{l\in\mathbb{Z}\setminus\{0\}}\hat{u}^2_l(t,y)\hat{\omega}^1_{-l}(t,y)=-f_0^2(t,y).\\
		\end{aligned}
	\end{equation}
	By employing integration by parts, H\"{o}lder's inequality, and Young's inequality in equation \eqref{5.15}, we obtain
	
	\[
	\begin{aligned}
		&\left\langle \left( \partial_t-\nu\partial^2_y\right) \bar{u}^1(t,y),-\partial^2_y\bar{u}^1\right\rangle =\dfrac{1}{2}\partial_t||\partial_y\bar{u}^1(t)||^2_{L^2}+\nu||\partial^2_y\bar{u}^1||^2_{L^2}=\left\langle f^2_0,\partial^2_y\bar{u}^1\right\rangle,\\
		%	&\leq||f^2_0||_{L^2}||\partial^2_y\bar{u}^1||_{L^2}\\
		%	&\leq  C\nu^{-1}||f^2_0||^2_{L^2},
	\end{aligned}
	\]
	which has  $$\partial_t||\partial_y\bar{u}^1(t)||^2_{L^2}+\nu||\partial^2_y\bar{u}^1||^2_{L^2} \leq  C\nu^{-1}||f^2_0||^2_{L^2},$$
	from the above equation, by combining it with $\partial_y\bar{u}^1(t,y)=\bar{\omega}(t,y),$ then we can deduce that
	$$E_0^2=||\bar{\omega}||^2_{L^\infty L^2}\leq C\nu^{-1}||f^2_0(t,y)||^2_{L^2L^2}+||\bar{\omega}_{0}||^2_{L^2}.$$
	The utilization of \eqref{5.8} allows us to derive the desired result
	$$E_0\leq C\nu^{-\frac{1}{2}}\sum_{l\in\mathbb{Z}}E_lE_{-l}+||\bar{\omega}_{0}||_{L^2}.$$
	Proof of \eqref{5.4}.
	The evolution equation of $\bar{\theta}$ can be derived in a similar manner, taking into account the influence of  $\nabla\cdot u=0$. Consequently, we obtain $u\cdot\nabla\theta=\nabla\cdot(u\theta)$ and 
	%\[
	%	\begin{aligned}
		\begin{equation}\label{5.16}
			\partial_t\bar{\theta}(t,y)-\mu\partial^2_y\bar{\theta}(t,y)=-\sum_{l\in\mathbb{Z}\setminus\{0\}}\partial_y(\hat{u}^2_l\hat{\theta}_{-l})(t,y)
			=-\partial_yg_0^2(t,y).
		\end{equation}
		By utilizing integration by parts and the H{\"o}lder inequality in equation \eqref{5.16}, we obtain
		\begin{equation}\label{5.17}
			\left\langle \left( \partial_t-\mu\partial^2_y\right) \bar{\theta}(t,y),\bar{\theta}(t,y)\right\rangle =\dfrac{1}{2}\partial_t||\bar{\theta}(t,y)||^2_{L^2}+\mu||\partial_y\bar{\theta}||^2_{L^2}=\left\langle g^2_0,\partial_y\bar{\theta}\right\rangle 
			\leq||g^2_0||_{L^2}||\partial_y\bar{\theta}||_{L^2},
		\end{equation}
		by employing integration with respect to t and Young's inequality in equation  \eqref{5.17}, then we can infer that
		\begin{equation}\label{5.18}
			H_0^2=||\bar{\theta}||^2_{L^\infty L^2}\leq C\mu^{-1}||g^2_0||^2_{L^2}+||\bar{\theta}_{0}||^2_{L^2}.
		\end{equation}
		According to  the Gagliardo-Nirenberg inequality and  the divergence-free condition $\partial_y\hat{u}^2_k=-ik\hat{u}^1_k,$
		we get
		$$||\hat{u}^2_k||_{L^2L^\infty}\leq C||\hat{u}^2_k||^{\frac{1}{2}}_{L^2L^2}||\partial_y\hat{u}^2_k||^{\frac{1}{2}}_{L^2L^2}=C|k|^{\frac{1}{2}}||\hat{u}^2_k||^{\frac{1}{2}}_{L^2L^2}||\hat{u}^1_k||^{\frac{1}{2}}_{L^2L^2}\leq CE_k,$$
		and then, we obtain
		\begin{equation}\label{5.19}
			||g^2_0||_{L^2L^2}\leq\sum_{l\in\mathbb{Z}\setminus\{0\}}||\hat{u}^2_l||_{L^2L^\infty}||\hat{\theta}_{-l}||_{L^\infty L^2}\leq\sum_{l\in\mathbb{Z}\setminus\{0\}}|-l|^{-\frac{1}{8}}E_lH_{-l}.
		\end{equation}
		Therefore, based on equations  $\eqref{5.18}$ and $\eqref{5.19},$  we can conclude that
		$$H_0\lesssim\mu^{-\frac{1}{2}}\sum_{l\in\mathbb{Z}\setminus\{0\}}|l|^{-\frac{1}{8}}E_lH_{-l}+||\bar{\theta}_{0}||_{L^2}.$$
		
		To control the nonlinear term $\|g^2_k\|_{L^2L^2}$ during the estimates of $H_k$, we need to divide
		them into the low-frequency part $\mu k^2 \leq 1$ and the high-frequency part $\mu k^2 > 1$.
		
			Proof of \eqref{5.5}. For the nonlinear problem for temperature, let us recall that
		\begin{equation}
		\left\{
		\begin{array}{l}
			\partial_t\hat{\theta}_k-\mu(\partial_y^2-|k|^2)\hat{\theta}_k+i k\left(1-y^2\right)\hat{\theta}_k=-i k \hat{g}_k^1-\partial_y \hat{g}^2_k,\\
			\hat{\theta}_k( \pm 1)=0,\left.\quad\hat{ \theta}_k \right|_{t=0}=\hat{\theta}_0\left(k, y\right).
		\end{array}
		\right.
	\end{equation}
		Denoting $$\hat{\theta}_k(t,y)=\dfrac{1}{2\pi}\int_{\mathbb{T}}\theta(t,x,y)e^{-ikx}dx,$$
		and $$g^1_k(t,y)=\sum_{l\in\mathbb{Z}}\hat{u}^1_l\hat{\theta}_{k-l}(t,y),\quad g^2_k(t,y)=\sum_{l\in\mathbb{Z}}\hat{u}^2_l\hat{\theta}_{k-l}(t,y),$$
		then we have 
		$$\partial_t\hat{\theta}_k(t,y)+\mu(k^2-\partial^2_y)\hat{\theta}_k(t,y)+ik(1-y^2)\hat{\theta}_k(t,y)=-ikg^1_k(t,y)-\partial_y g^2_k(t,y).$$
		The Proposition \ref{pro4.6} implies that
		\begin{equation}\label{5.21}
			H_k\leq |k|^{\frac{1}{8}}||\hat{\theta}_{0}(k,y)||_{L^2}+C\left(  \mu^{-\frac{3}{8}}|k|^{\frac{1}{2}} ||g^1_k||_{L^2L^2}+\mu^{-\frac{1}{2}}|k|^{\frac{1}{8}}||g^2_k||_{L^2L^2}\right),
		\end{equation}
			where we have used the fact that $$|k| \min \left\{(\nu|k|)^{-\frac{1}{4}},\left(\nu|k|^2\right)^{-\frac{1}{2}}\right\} \leq \nu^{-\frac{3}{8}}|k|^{\frac{3}{8}}.$$
		Noticing that $\bar{u}^2=0$, we have that for $k\neq 0,$
		\begin{equation}\label{5.22}
			\begin{aligned}
				||g^2_k||_{L^2L^2}\leq& \sum_{l\in\mathbb{Z}}\|\hat{u}^2_l\hat{\theta}_{k-l}(t,y)\|_{L^2L^2}\\
				\leq
				&	\sum_{l\in\mathbb{Z} } ||\hat{u}^2_l||_{L^2L^\infty}||\hat{\theta}_{k-l}||_{L^\infty L^2}\\
				\leq&\sum_{l\in\mathbb{Z} } ||\hat{u}^2_l||_{L^2L^2}^{\frac{1}{2}}||\partial_y\hat{u}^2_l\|_{L^2L^2}^{\frac{1}{2}}||\hat{\theta}_{k-l}||_{L^\infty L^2}\\
				\leq&\sum_{l\in\mathbb{Z} } ||\hat{u}^2_l||_{L^2L^2}^{\frac{1}{2}}||l\hat{u}^1_l\|_{L^2L^2}^{\frac{1}{2}}||\hat{\theta}_{k-l}||_{L^\infty L^2}\\
				\leq&\sum_{l\in\mathbb{Z} } E_lH_{k-l}.\\
			\end{aligned}
		\end{equation}
		In addition to, for $g^1_k$ and $k\neq 0$
		\begin{equation}\label{5.23}
			\begin{aligned}
				||g^1_k||_{L^2L^2}&\leq ||\bar{u}^1||_{L^\infty L^\infty}||\hat{\theta}_{k}||_{L^2 L^2}+ ||\hat{u}^1_k||_{L^2 L^\infty}||\bar{\theta}||_{L^\infty L^2}+\sum_{l\in\mathbb{Z},l\neq \left\lbrace 0,k\right\rbrace } ||\hat{u}^1_l\hat{\theta}_{k-l}||_{L^2 L^2}\\
				&\leq  ||\bar{\omega}||_{L^\infty L^2}||\hat{\theta}_{k}||_{L^2 L^2}+\nu^{-\frac{1}{8}}|k|^{-\frac{3}{8}}E_kH_0+\sum_{l\in\mathbb{Z},l\neq \left\lbrace 0,k\right\rbrace } ||\hat{u}^1_l\hat{\theta}_{k-l}||_{L^2 L^2}\\
				&\leq|k|^{-\frac{1}{8}}	(\mu|k|)^{-\frac{1}{4} \cdot \frac{5}{6}}\left(\mu^{\frac{1}{2}}|k|\right)^{-\frac{1}{6}} E_0H_k+\nu^{-\frac{1}{8}}|k|^{-\frac{3}{8}}E_kH_0+\sum_{l\in\mathbb{Z},l\neq \left\lbrace 0,k\right\rbrace } ||\hat{u}^1_l\hat{\theta}_{k-l}||_{L^2 L^2}.\\
					&\leq	\mu^{-\frac{7}{24}}|k|^{-\frac{1}{2}} E_0H_k+\nu^{-\frac{1}{8}}|k|^{-\frac{3}{8}}E_kH_0+\sum_{l\in\mathbb{Z},l\neq \left\lbrace 0,k\right\rbrace } ||\hat{u}^1_l\hat{\theta}_{k-l}||_{L^2 L^2}.\\
			\end{aligned}
		\end{equation}
		The following step involves estimating $$\sum_{l\in\mathbb{Z},l\neq \left\lbrace 0,k\right\rbrace } ||\hat{u}^1_l\hat{\theta}_{k-l}||_{L^2 L^2},$$ 
				we divide it into two parts and get
		\[
		\begin{aligned}
			\sum_{l\in\mathbb{Z}\setminus \left\lbrace 0,k\right\rbrace } ||\hat{u}^1_l\hat{\theta}_{k-l}||_{L^2 L^2}&\leq\sum_{l\in\mathbb{Z}\setminus \left\lbrace 0,k\right\rbrace,|k-l|\leq\frac{|k|}{2}} ||\hat{u}^1_l\hat{\theta}_{k-l}||_{L^2 L^2}+\sum_{l\in\mathbb{Z}\setminus \left\lbrace 0,k\right\rbrace,|k-l|\geq\frac{|k|}{2}} ||\hat{u}^1_l\hat{\theta}_{k-l}||_{L^2 L^2},\\
			&=\text{HL}+\text{LH}.
		\end{aligned}
		\]
		For HL,
		\begin{equation}\label{5.24}
			\begin{aligned}
				\text{	HL}	&\leq\sum_{l\in\mathbb{Z}\setminus\left\lbrace 0,k\right\rbrace, |k-l|\leq\frac{|k|}{2}} ||\hat{u}^1_l||_{L^2 L^\infty}||\hat{\theta}_{k-l}||_{L^\infty L^2}\\
				&\lesssim\sum_{l\in\mathbb{Z}\setminus \left\lbrace 0,k\right\rbrace ,|k-l|\leq\frac{|k|}{2}} \nu^{-\frac{1}{8}}|l|^{-\frac{3}{8}}E_l|k-l|^{-\frac{1}{8}}H_{k-l}\\
				&\lesssim\sum_{l\in\mathbb{Z}\setminus \left\lbrace 0,k\right\rbrace  ,|k-l|\leq\frac{|k|}{2}} \nu^{-\frac{1}{8}}|k|^{-\frac{3}{8}}E_lH_{k-l}\\
			\end{aligned}
		\end{equation}
		and 
		\begin{equation}\label{5.25}
			\begin{aligned}
				\text{	LH}	&\leq\sum_{l\in\mathbb{Z}\setminus \left\lbrace 0,k\right\rbrace  ,|k-l|\geq\frac{|k|}{2}} ||\hat{u}^1_l||_{L^\infty L^\infty}||\hat{\theta}_{k-l}||_{L^2 L^2}\\
				&\lesssim\sum_{l\in\mathbb{Z}\setminus\left\lbrace 0,k\right\rbrace ,|k-l|\geq\frac{|k|}{2}} |l|^{-\frac{1}{2}}|k-l|^{-\frac{1}{8}}(\mu|k-l|)^{-\frac{1}{4} \cdot \frac{5}{6}}\left(\mu^{\frac{1}{2}}|k-l|\right)^{-\frac{1}{6}}E_lH_{k-l}\\
				&\lesssim\sum_{l\in\mathbb{Z}\setminus \left\lbrace 0,k\right\rbrace ,|k-l|\geq\frac{|k|}{2}} \mu^{-\frac{7}{24}}|k|^{-\frac{1}{2}}E_lH_{k-l}.\\
			\end{aligned} 
		\end{equation}
			And then, substituting \eqref{5.25} and \eqref{5.24} into \eqref{5.23}, we get
		\begin{equation}\label{5.26}
			\begin{aligned}
				||g^1_k||_{L^2L^2}&\lesssim|k|^{-\frac{1}{2}}\mu^{-\frac{7}{24}}E_0H_k+\nu^{-\frac{1}{8}}|k|^{-\frac{3}{8}}E_kH_0\\
				&+\sum_{l\in\mathbb{Z}\setminus\left\lbrace 0,k\right\rbrace ,|k-l|\leq\frac{|k|}{2}} \nu^{-\frac{1}{8}}|k|^{-\frac{3}{8}}E_lH_{k-l}+\sum_{l\in\mathbb{Z}\setminus \left\lbrace 0,k\right\rbrace,|k-l|\geq\frac{|k|}{2}} \mu^{-\frac{7}{24}}|k|^{-\frac{1}{2}}E_lH_{k-l}.\\
			\end{aligned}
		\end{equation}
		The combination of equations  \eqref{5.21}, \eqref{5.22} and \eqref{5.26} yields the result for $k\neq 0,$ and $\mu k^2\leq 1$
		\[
		\begin{aligned}
			&	H_k\leq|k|^{\frac{1}{8}}||\hat{\theta}_{0}(k,y)||_{L^2}+C\left(  \mu^{-\frac{3}{8}}|k|^{\frac{1}{2}} ||g^1_k||_{L^2L^2}+\mu^{-\frac{1}{2}}|k|^{\frac{1}{8}}||g^2_k||_{L^2L^2}\right)\\
			&\lesssim|k|^{\frac{1}{8}}||\hat{\theta}_{0}(k,y)||_{L^2}+\mu^{-\frac{2}{3}}E_0H_k+\mu^{-\frac{3}{8}}\nu^{-\frac{1}{8}}|k|^{\frac{1}{8}}E_kH_0\\
			&+\sum_{l\in\mathbb{Z}\setminus\left\lbrace 0,k\right\rbrace ,|k-l|\leq\frac{|k|}{2}} \mu^{-\frac{3}{8}}\nu^{-\frac{1}{8}}|k|^{\frac{1}{8}}E_lH_{k-l}+\sum_{l\in\mathbb{Z}\setminus \left\lbrace 0,k\right\rbrace,|k-l|\geq\frac{|k|}{2}} \mu^{-\frac{2}{3}}E_lH_{k-l}+\mu^{-\frac{1}{2}}|k|^{\frac{1}{8}}\sum_{l\in\mathbb{Z} } E_lH_{k-l}\\
			&\lesssim|k|^{\frac{1}{8}}||\hat{\theta}_{0}(k,y)||_{L^2} +\nu^{-\frac{1}{8}}\mu^{-\frac{7}{16}}\sum_{l\in\mathbb{Z}\setminus\left\lbrace 0,k\right\rbrace ,|k-l|\leq\frac{|k|}{2}} E_lH_{k-l}\\
			&+\sum_{l\in\mathbb{Z}\setminus \left\lbrace 0,k\right\rbrace,|k-l|\geq\frac{|k|}{2}} \mu^{-\frac{2}{3}}E_lH_{k-l}+\mu^{-\frac{9}{16}}\sum_{l\in\mathbb{Z} } E_lH_{k-l}.\\
			&\lesssim|k|^{\frac{1}{8}}||\hat{\theta}_{0}(k,y)||_{L^2}+\sum_{l\in\mathbb{Z}} \mu^{-\frac{2}{3}}E_lH_{k-l}+\nu^{-\frac{1}{8}}\mu^{-\frac{7}{16}}\sum_{l\in\mathbb{Z}\setminus\left\lbrace 0,k\right\rbrace ,|k-l|\leq\frac{|k|}{2}} E_lH_{k-l}\\
		\end{aligned}
		\]
		Proof of \eqref{5.6}.
		For $\mu k^2\geq 1,$ form  Proposition \ref{pro4.2}, we get
		\begin{equation}\label{5.27}
			\begin{aligned}
				&H_k=|k|^{\frac{1}{8}}||\hat{\theta}_k||_{L^\infty L^2}+\mu^{\frac{1}{4}} k^{\frac{3}{8}}||\hat{\theta}_k||_{L^2L^2}+\mu^{\frac{1}{2}} |k|^{\frac{9}{8}}||\hat{\theta}_k||_{L^2L^2}\\
				&=|k|^{\frac{1}{8}}\left( ||\hat{\theta}_k||_{L^\infty L^2}+|k|^{-\frac{2}{8}}(\mu k^2)^{\frac{1}{4}} ||\hat{\theta}_k||_{L^2L^2}+\mu^{\frac{1}{2}} |k|||\hat{\theta}_k||_{L^2L^2} \right)\\
				&\leq|k|^{\frac{1}{8}}\left( ||\hat{\theta}_k||_{L^\infty L^2}+(\mu k^2)^{\frac{1}{2}} ||\hat{\theta}_k||_{L^2L^2}+\mu^{\frac{1}{2}} |k|||\hat{\theta}_k||_{L^2L^2}\right) \\
				&\lesssim|k|^{\frac{1}{8}}\left( ||\hat{\theta}_0(k,y)||_{ L^2}+\mu ^{-\frac{1}{2}} \left( ||g^1_k||_{L^2L^2}+||g^2_k||_{L^2L^2}\right) \right), \\
			\end{aligned}
		\end{equation}
		and then we obtain that $k\neq 0$
			In addition to, for $g^1_k$ and $k\neq 0$
		\begin{equation}\label{5.28}
			\begin{aligned}
				||g^1_k||_{L^2L^2}&\leq ||\bar{u}^1||_{L^\infty L^\infty}||\hat{\theta}_{k}||_{L^2 L^2}+ ||\hat{u}^1_k||_{L^2 L^\infty}||\bar{\theta}||_{L^\infty L^2}+\sum_{l\in\mathbb{Z},l\neq \left\lbrace 0,k\right\rbrace } ||\hat{u}^1_l\hat{\theta}_{k-l}||_{L^2 L^2}\\
				&\leq  ||\bar{\omega}||_{L^\infty L^2}||\hat{\theta}_{k}||_{L^2 L^2}+\nu^{-\frac{1}{8}}|k|^{-\frac{3}{8}}E_kH_0+\sum_{l\in\mathbb{Z},l\neq \left\lbrace 0,k\right\rbrace } ||\hat{u}^1_l\hat{\theta}_{k-l}||_{L^2 L^2}\\
				&\leq|k|^{-\frac{1}{8}}	(\mu|k|)^{-\frac{1}{4} \cdot \frac{5}{6}}\left(\mu^{\frac{1}{2}}|k|\right)^{-\frac{1}{6}} E_0H_k+\nu^{-\frac{1}{8}}|k|^{-\frac{3}{8}}E_kH_0+\sum_{l\in\mathbb{Z},l\neq \left\lbrace 0,k\right\rbrace } ||\hat{u}^1_l\hat{\theta}_{k-l}||_{L^2 L^2}.\\
				&\leq	\mu^{-\frac{7}{24}}|k|^{-\frac{1}{2}} E_0H_k+\nu^{-\frac{1}{8}}|k|^{-\frac{3}{8}}E_kH_0+\sum_{l\in\mathbb{Z},l\neq \left\lbrace 0,k\right\rbrace } ||\hat{u}^1_l\hat{\theta}_{k-l}||_{L^2 L^2}.\\
			\end{aligned}
		\end{equation}
		Whereas for the term $$\sum_{l\in\mathbb{Z},l\neq \left\lbrace 0,k\right\rbrace } ||\hat{u}^1_l\hat{\theta}_{k-l}||_{L^2 L^2},$$
		and the fact $|k|\lesssim|l||k-l|(l\neq\left\lbrace 0,k\right\rbrace )$, we can obtain
			\begin{equation}\label{5.29}
			\begin{aligned}
				\text{	LH}	&\leq\sum_{l\in\mathbb{Z}\setminus \left\lbrace 0,k\right\rbrace  } ||\hat{u}^1_l||_{L^\infty L^\infty}||\hat{\theta}_{k-l}||_{L^2 L^2}\\
				&\lesssim\sum_{l\in\mathbb{Z}\setminus\left\lbrace 0,k\right\rbrace } |l|^{-\frac{1}{2}}|k-l|^{-\frac{1}{8}}(\mu|k-l|)^{-\frac{1}{4} \cdot \frac{5}{6}}\left(\mu^{\frac{1}{2}}|k-l|\right)^{-\frac{1}{6}}E_lH_{k-l}\\
				&\lesssim\sum_{l\in\mathbb{Z}\setminus \left\lbrace 0,k\right\rbrace } \mu^{-\frac{7}{24}}|k|^{-\frac{1}{2}}E_lH_{k-l}.\\
			\end{aligned} 
\end{equation}
And then, we obtain 

	\begin{equation}\label{5.30}
	\begin{aligned}
		||g^1_k||_{L^2L^2}&\lesssim|k|^{-\frac{1}{2}}\mu^{-\frac{7}{24}}E_0H_k+\nu^{-\frac{1}{8}}|k|^{-\frac{3}{8}}E_kH_0
		+\sum_{l\in\mathbb{Z}\setminus \left\lbrace 0,k\right\rbrace} \mu^{-\frac{7}{24}}|k|^{-\frac{1}{2}}E_lH_{k-l}.\\
	\end{aligned}
\end{equation}
For the term $$g^2_k(t,y)=\sum_{l\in\mathbb{Z}}\hat{u}^2_l\hat{\theta}_{k-l}(t,y).$$
Noticing that $\bar{u}^2=0$, we have that for $k\neq 0,$
		\begin{equation}\label{5.31}
			\begin{aligned}
				||g^2_k||_{L^2L^2}\leq& \sum_{l\in\mathbb{Z},}\|\hat{u}^2_l\hat{\theta}_{k-l}\|_{L^2L^2}+\|\hat{u}^2_k\|_{L^2L^\infty}\|\bar{\theta}\|_{L^\infty L^2}\\
				\leq
				&\sum_{l\in\mathbb{Z} ,l\neq \left\lbrace 0,k\right\rbrace } ||\hat{u}^2_l||_{L^\infty L^\infty}||\hat{\theta}_{k-l}||_{L^2 L^2}+ ||\hat{u}^2_k||_{L^2 L^2}^{\frac{1}{2}}||k\hat{u}^1_k\|_{L^2L^2}^{\frac{1}{2}}H_0\\
				\leq&\sum_{l\in\mathbb{Z} } ||\hat{u}^2_l||_{L^\infty L^2}^{\frac{1}{2}}||\partial_y\hat{u}^2_l\|_{L^\infty L^2}^{\frac{1}{2}}||\hat{\theta}_{k-l}||_{L^2 L^2}+|k|^{-\frac{3}{4}}\nu^{-\frac{1}{4}}E_kH_0\\
				\leq&\sum_{l\in\mathbb{Z} } ||\hat{u}^2_l||_{L^\infty L^2}^{\frac{1}{2}}||l\hat{u}^1_l\|_{L^\infty L^2}^{\frac{1}{2}}||\hat{\theta}_{k-l}||_{L^2 L^2}+|k|^{-\frac{3}{4}}\nu^{-\frac{1}{4}}E_kH_0\\
				\leq&\sum_{l\in\mathbb{Z} }|l|^{-\frac{1}{2}}\mu^{-\frac{1}{4}} |k-l|^{-\frac{3}{8}}E_lH_{k-l}+|k|^{-\frac{3}{4}}\nu^{-\frac{1}{4}}E_kH_0.\\
				\leq&\sum_{l\in\mathbb{Z} }\mu^{-\frac{1}{4}} |k|^{-\frac{3}{8}}E_lH_{k-l}+|k|^{-\frac{3}{4}}\nu^{-\frac{1}{4}}E_kH_0.\\
			\end{aligned}
		\end{equation}
						Combining \eqref{5.27}, \eqref{5.30} and \eqref{5.31}, we get 
						for $\mu k^2\geq 1,$ 
						\[
						\begin{aligned}
							H_k&
							\leq|k|^{\frac{1}{8}}\left( ||\hat{\theta}_0(k,y)||_{ L^2}+\mu ^{-\frac{1}{2}} \left( ||g^1_k||_{L^2L^2}+||g^2_k||_{L^2L^2}\right) \right), \\
							&\lesssim|k|^{\frac{1}{8}}||\hat{\theta}_{0}(k,y)||_{L^2}+|k|^{-\frac{3}{8}}\mu^{-\frac{19}{24}}E_0H_k+\mu^{-\frac{1}{2}}\nu^{-\frac{1}{8}}|k|^{-\frac{1}{4}} E_kH_0+\mu^{-\frac{1}{2}}\nu^{-\frac{1}{4}}|k|^{-\frac{5}{8}} E_kH_0\\
							&+\sum_{l\in\mathbb{Z}\setminus \left\lbrace 0,k\right\rbrace} \mu^{-\frac{19}{24}}|k|^{-\frac{3}{8}}E_lH_{k-l}+\mu^{-\frac{3}{4}}|k|^{-\frac{1}{4}}\sum_{l\in\mathbb{Z} } E_lH_{k-l}\\
							%	&+|k|^{-\frac{1}{3}}\mu^{-\frac{2}{3}}\sum_{l\in\mathbb{Z},l\neq \left\lbrace 0,k\right\rbrace ,|k-l|\geq\frac{|k|}{2}} E_lH_{k-l}+\sum_{l\in\mathbb{Z},l\neq \left\lbrace 0,k\right\rbrace ,|k-l|\leq\frac{|k|}{2}} C\mu^{-\frac{1}{2}}\nu^{-\frac{1}{12}}|k|^{-\frac{1}{3}}E_lH_{k-l}\\
							&\lesssim|k|^{\frac{1}{8}}||\hat{\theta}_{0}(k,y)||_{L^2}+\mu^{-\frac{9}{16}}\sum_{l\in\mathbb{Z}\setminus\{0\}} E_lH_{k-l}+C\mu^{-\frac{5}{8}}\sum_{l\in\mathbb{Z}}E_lH_{k-l}+\mu^{-\frac{3}{8}}\nu^{-\frac{1}{8}} E_kH_0\\
								&\lesssim|k|^{\frac{1}{8}}||\hat{\theta}_{0}(k,y)||_{L^2}+C\mu^{-\frac{5}{8}}\sum_{l\in\mathbb{Z}}E_lH_{k-l}+\mu^{-\frac{3}{8}}\nu^{-\frac{1}{8}} E_kH_0\\
						\end{aligned}
						\]
						This completes the proof of Proposition \ref{pro5.1}.
					\end{proof}
					Theorem \ref{thm1} is now proven. By utilizing \eqref{5.2} and \eqref{5.3}, we derive
					\begin{equation}\label{5.32}
						\sum_{k\in\mathbb{Z}}E_k\leq\sum_{k\in\mathbb{Z}\setminus\{0\}}||\Delta_k\hat{\omega}_{0}(k,y)||_{L^2}+\|\bar{\omega}_0\|_{L^2}+\sum_{k\in\mathbb{Z}\setminus\{0\} }C\nu^{-\frac{3}{8}}\mu^{-\frac{1}{4}}H_k+C\nu^{-\frac{2}{3}}\sum_{k\in\mathbb{Z}}\sum_{l\in\mathbb{Z}}E_lE_{k-l} .
					\end{equation}
					And by the fact 
					$$\sum_{k\in\mathbb{Z}}H_k=H_0+\sum_{k\in\mathbb{Z}\setminus\{0\}, \mu k^2\leq 1}H_k+\sum_{k\in\mathbb{Z}\setminus\{0\},\mu k^2>1}H_k,$$
					combining \eqref{5.4}, \eqref{5.5} and \eqref{5.6}, we can deduce
					\begin{equation}\label{5.33}
						\begin{aligned}
							&\sum_{k\in\mathbb{Z}}H_k\lesssim||\bar{\theta}_{0}||_{L^2}+\sum_{k\in\mathbb{Z}\setminus\{0\} }|k|^{\frac{1}{8}}||\hat{\theta}_{0}(k,y)||_{L^2} +\mu^{-\frac{2}{3}}\sum_{k\in\mathbb{Z} }\sum_{l\in\mathbb{Z} } E_lH_{k-l}\\
							&+\nu^{-\frac{1}{8}}\mu^{-\frac{7}{16}}\sum_{k\in\mathbb{Z}\setminus\{0\}, \mu k^2\leq 1}\sum_{l\in\mathbb{Z}} E_lH_{k-l}
							+\mu^{-\frac{3}{8}}\nu^{-\frac{1}{8}}\sum_{k\in\mathbb{Z}\setminus\{0\},\mu k^2>1}E_kH_{0}.\\
						\end{aligned}
					\end{equation}
					On the other hand, if 	$$||u_{0}||_{H^{\frac{7}{2}+}}\leq c_0\min\left\lbrace \mu,\nu\right\rbrace ^{\frac{2}{3}},\quad ||\theta_{0}||_{H^1}+|||D_x|^{\frac{1}{8}}\theta_{0}||_{H^1}\leq c_1\min\left\lbrace \mu,\nu\right\rbrace ^{{\frac{31}{24}}},$$
					we have the ability to obtain
					$$\sum_{k\in\mathbb{Z}\setminus\{0\}}||\Delta_k\hat{\omega}_{0}(k,y)||_{L^2}+\|\bar{\omega}_0\|_{L^2}\leq Cc_0\min\left\lbrace \mu,\nu\right\rbrace ^{\frac{2}{3}},$$
					and 
					$$||\bar{\theta}_{0}||_{L^2}+\sum_{k\in\mathbb{Z}\setminus\{0\}}|k|^{\frac{1}{8}}||\hat{\theta}_{0}(k,y)||_{L^2} \leq Cc_1\min\left\lbrace \mu,\nu\right\rbrace ^{\frac{31}{24}}.$$
					Thus, we select $c_0,c_1$ suitably small, by bootstrap arguments,  form \eqref{5.32} and \eqref{5.33}, then we can infer that
					
					$$\sum_{k\in\mathbb{Z}}E_k\leq Cc_0\min\left\lbrace \mu,\nu\right\rbrace ^{\frac{2}{3}},\quad\sum_{k\in\mathbb{Z}}H_k\leq Cc_1\min\left\lbrace \mu,\nu\right\rbrace ^{\frac{31}{24}}.$$
					This completes the proof of Theorem \ref{thm1}.
					
					\textbf{Conflict of Interest}\quad The authors declared that they have no conflict of interest.
					\section{Acknowledgment}
				I appreciate my supervisor Prof. Weike Wang  for suggesting this problem and many valuable discussions.  The work of Gaofeng Wang is supported by National Nature Science Foundation of China 12271357 and 12161141004 and Shanghai Science and Technology Innovation Action
				Plan No. 21JC1403600.
					

\begin{thebibliography}{99}
						\bibitem{ref39} D. Albritton, R. Beekie, M. Novack, Enhanced dissipation and Hörmander’s hypoellipticity, Journal of Functional Analysis, 283 (3) (2022) 109522.
						\bibitem{ref31}
						D. Adhikari, C. Cao, H. Shang, J. Wu, X. Xu, Z. Ye, Global regularity results for the 2D Boussinesq
						equations with partial dissipation, J. Differ. Equ. 260 (2016) 1893–1917.
						\bibitem{ref30}
						H. Abidi, T. Hmidi, On the global well-posedness for Boussinesq system, J. Differ. Equ. 233 (2007)
						199–220.
						\bibitem{ref40}A.J. Bernoff, J.F. Lingevitch, Rapid relaxation of an axisymmetric vortex, Phys Fluids. 6 (1994) 3717–3723.
						
						\bibitem{ref8}D. Bian, X. Pu, Stability threshold for 2D shear flows of the Boussinesq system near Couette, J.Math. Phys. 63 (8) (2022) 081501.
						
					   \bibitem{Bedrossian.2015}J. Bedrossian and N. Masmoudi, Inviscid damping and the asymptotic stability of planar shear flow in the 2D Euler equations, Publ. Math. Inst. Hautes ´Etudes Sci. 122 (2015) 195-300.
					   
						\bibitem{48}J. Bedrossian, P. Germain, N. Masmoudi,  Stability of the Couette flow at high
						Reynolds number in 2D and 3D, Bull. Am. Math. Soc. 56 (3) (2019) 373–414.
						
						\bibitem{ref20}J. Bedrossian, P. Germain,  N. Masmoudi,  Dynamics near the subcritical transition of the 3D Couette flow I: Below threshold case, Mem. Amer. Math. Soc. 266 (2020) no. 1294.
						\bibitem{Bedrossian.2022}J. Bedrossian, P. Germain,  N. Masmoudi,  Dynamics near the subcritical transition of the 3D Couette flow II: Above  threshold case, Mem. Amer. Math. Soc. 279 (2022) no. 1377.
						\bibitem{ref51}J. Bedrossian, P. Germain,  N. Masmoudi, On the stability threshold for the 3D Couette flow in Sobolev regularity, Ann, Math. 185 (2) (2017) 541–608.
						
						\bibitem{ref22}J. Bedrossian, N.  Masmoudi, V. Vicol, Enhanced dissipation and inviscid damping in the inviscid limit of the Navier-Stokes equations near the two dimensional Couette flow, Arch. Ration. Mech. Anal. 219 (3) (2016) 1087–1159.
						\bibitem{ref23}J. Bedrossian,  F. Wang, V. Vicol,  The Sobolev stability threshold for 2D shear flows near Couette, J. Nonlinear Sci. 28 (6) (2018) 2051–2075.
						
						\bibitem{ref5}A. Castro, D. Córdoba, D. Lear, On the asymptotic stability of stratified solutions for the 2D
						Boussinesq equations with a velocity damping term, Math. Models Methods Appl. Sci. 29 (2019) 1227–1277.
						\bibitem{ref34}
					M. Coti Zelati, T.M. Elgindi, K. Widmayer, Enhanced dissipation in the Navier-Stokes equations near the Poiseuille flow, Commun. Math. Phys. 378 (2) (2020) 987–1010.
						\bibitem{ref1}	P. Constantin, C. Doering, Heat transfer in convective turbulence, Nonlinearity 9 (1996) 1049–1060.
						\bibitem{ref44}
					 P. Constantin, A. Kiselev, L. Ryzhik, A. Zlatos, Diffusion and mixing in fluid flow, Ann. Math. 168 (2) (2008) 643–674.
					
						\bibitem{ref52}
						Q. Chen, S. Ding, Z. Lin, Z. Zhang, Nonlinear stability for 3-D plane Poiseuille flow in a
						finite channel, arXiv preprint, arXiv:2310.11694V2, 2024.
						
						\bibitem{ref25}Q. Chen, T. Li, D. Wei, Z. Zhang, Transition threshold for the 2-D Couette flow in a finite channel, Arch. Ration. Mech. Anal. 238 (1) (2020) 125–183.
						
						\bibitem{ref26}Q. Chen, D. Wei, Z. Zhang, Transition threshold  for the 3D Couette flow in a finite channel, arXiv:2006.00721x1, 2020.
						\bibitem{ref36} Q. Chen, D. Wei, Z. Zhang, Linear stability of pipe Poiseuille flow at high Reyholds number regime, Communications on Pure and Applied Mathematics. 76 (9) (2023)  1868-1964.
						
					\bibitem{46} A. Del Zotto, Enhanced dissipation and transition threshold for the Poiseuille flow in a periodic strip, SIAM J. Math. Anal., 55(2023),4410-4424.
						\bibitem{ref2} C. Doering, J. Gibbon, Applied Analysis of the Navier-Stokes Equations, Cambridge Texts in Applied Mathematics, Cambridge University Press, Cambridge, 1995.
						\bibitem{ref35}
					S. Ding, Z. Lin,  Enhanced dissipation and transition threshold for the 2-D plane Poiseuille flow via resolvent estimate, Journal of Differential Equation. 332 (2022) 404-439.
						\bibitem{ref53}
					S. Ding, Z. Lin, Stability for the 2D plane poiseuille flow in finite channel, arXiv:2401.00417v2, 2024.
						\bibitem{ref9}W. Deng, J. Wu, P. Zhang, Stability of Couette flow for 2D Boussinesq system with vertical dissipation, J. Funct. Anal. 281 (12) (2021) 109255.
						\bibitem{ref15}Y. Duguet, L. Brandt, B.  Larsson, Towards minimal perturbations in transitional
						plane Couette flow, Phys, Rev. E(3). 82 (2) (2010) 026316, 13.
					\bibitem{G1931}0S. Goldstein, “On the stability of superposed streams of fluids of different densities,” Proc. R. Soc. London, Ser. A 132(820), 524–548 (1931).
						\bibitem{ref38}L. Hörmander, Hypoelliptic second order differential equations, Acta Math. 119 (1967) 147–171.
						\bibitem{ref32}
						T. Hmidi, S. Keraani, On the global well-posedness of the two-dimensional Boussinesq system with
						a zero diffusivity, Adv. Differ. Equ. 12 (2007) 461–480.
						\bibitem{ref33}
						T. Hou, C. Li, Global well-posedness of the viscous Boussinesq equations, Discrete Contin. Dyn.
						Syst. 12 (2005) 1–12.
						\bibitem{H1961}
						L.N. Howard, Note on a paper of John W. Miles, J. Fluid Mech., 10 (1961), pp. 509--512.
						\bibitem{Ionescu.2019}
						A.D. Ionescu and H. Jia. Inviscid damping near the Couette flow in a channel. Communications in
						Mathematical Physics. 374 (2020) 2015–2096.
						
						\bibitem{ref27}T. Kato, Perturbation Theory for Linear Operators, Die Grundlehren der Mathematischen Wissenschaften, vol. 132,Springer, New York, New York, 1966.
						
						\bibitem{ref16}A. Lundbladh, D. Henningson, S. Reddy, Threshold Amplitudes for Transition in Channel Flows, in Transition, pp. 309–318. Springer, New York 1994.
						
					\bibitem{ref50}H. Li, N. Masmoudi, and W. Zhao, Asymptotic stability of two-dimensional Couette flow in a viscous fluid, arXiv:2208.14898, 2022.
							
						\bibitem{ref41}M. Latini, A.J. Bernoff, Transient anomalous diffusion in Poiseuille flow, J. Fluid Mech. 441 (2001)
						399–411.
						\bibitem{ref45}T. Li, D. Wei, Z. Zhang, Pseudospectral bound and transition threshold for the 3D Kolmogorov flow, Commun. Pure
						Appl. Math. 73 (3) (2020) 465–557.
						
						
						\bibitem{ref3} A. Majda, Introduction to PDEs and Waves for the Atmosphere and Ocean, Courant Lecture Notes,
						vol. 9, Courant Institute of Mathematical Sciences and American Mathematical Society, 2003.
						
					\bibitem{ref49}A. Majda, A. Bertozzi, Vorticity and Incompressible Flow, Cambridge University Press, 2002.
					\bibitem{Zhao.2020}
					N. Masmoudi and W. Zhao, Nonlinear inviscid damping for a class of monotone shear flows in finite
					channel, arXiv:2001. 08564, 2020.
						\bibitem{ref14}N. Masmoudi, Cui. Zhai, W. Zhao,  Asymptotic stability for two-dimensional Boussinesq systems around the Couette flow in a finite channel, J. Funct. Anal. 284 (2023).
						
						\bibitem{ref13}N. Masmoudi, W. Zhao, Stability threshold of two-dimensional Couette flow in Sobolev spaces, Ann.Inst. Henri Poincaré, Anal. Non Linéaire 39 (2) (2022) 245–325.
						
					\bibitem{ref47} N. Masmoudi, W. Zhao, Enhanced dissipation for the 2D Couette flow in critical space, Communications in Partial Differential Equations. 45 (12) (2020) 1682-1701.
						
						\bibitem{ref17}S. Orszag, L. Kells, Transition to turbulence in plane Poiseuille and plane Couette flow, J. Fluid Mech. 96 (1980)
						159–205.
						
						\bibitem{ref55}
					W. Orr, The stability or instability of steady motions of a perfect liquid and of a
					viscous liquid. Part I: A perfect liquid. Proc. R. Irish Acad. Sec. A: Math. Phys.
					Sci. 27 (1907) 9–68.
					
						\bibitem{ref28}A. Pazy, Semigroups of Linear Operators and Applications to Partial Differential Equations, Applied Mathematical
						Sciences, vol. 44, Springer, New York 1983.
						
						\bibitem{ref4}J. Pedlosky, Geophysical Fluid Dynamics, Springer, New York, 1987.
						
						\bibitem{ref43}
					P.B. Rhines, W.R. Young, How rapidly is a passive scalar mixed within closed streamlines? J. Fluid
					Mech. 133 (1983) 133–145.
						
						\bibitem{ref18}S. Reddy, P. Schmid, J. Baggett, D.  Henningson,  On stability of streamwise
						streaks and transition thresholds in plane channel flows, J. Fluid Mech. 365 (1998) 269–303.
						\bibitem{Synge1933}J.L. Synge, “The stability of heterogeneous fluids,” Trans. R. Soc. Canada 27 (1993) 1–18.
						
						\bibitem{T1931}G.I. Taylor, “Effect of variation in density on the stability of superposed streams of fluid,” Proc. R. Soc. London, Ser. A 132 (820)(1931), 499–523.
						
						\bibitem{ref6}L. Tao, J. Wu, K. Zhao, X. Zheng, Stability near hydrostatic equilibrium to the 2D Boussinesq
						equations without thermal diffusion, Arch. Ration. Mech. Anal. 237 (2020) 585–630.
						\bibitem{ref42}
					W. Thomson, Stability of fluid motion-rectilinear motion of viscous fluid between two parallel plates,
					Philos. Mag. 24 (1887) 188–196.
						
							\bibitem{ref7}R. Wan, Global well-posedness for the 2D Boussinesq equations with a velocity damping term,
						Discrete Contin. Dyn. Syst. 39 (2019) 2709–2730.
						
						\bibitem{ref29}D. Wei, Diffusion and mixing in fluid flow via the resolvent estimate, Sci. China Math. (2019) 1-12.
						
						\bibitem{ref24}D. Wei, Z. Zhang, Transition threshold for the 3D Couette flow in Sobolev space. Comm. Pure Appl. Math. 74 (11) (2021) 2398–2479.
							\bibitem{ref37}Y. Wang, C. Xie, Uniform structural stability of Hagen-Poiseuille flows in a pipe, Communications in Mathematical Physics. 393 (3) (2022) 1347–1410.
							
								\bibitem{ref19}A. Yaglom,  Hydrodynamic Instability and Transition to Turbulence, Fluid Mech. Appl. 100. Springer, New York 2012.
						\bibitem{Zhao}C. Zhai, W. Zhao, Stability threshold of the Couette flow for Navier-Stokes Boussinesq system with large richardson number $\gamma^2\geq\frac{1}{4},$ SIAM J. Math. Anal. 55 (2) (2023) 1284-1318.
						\bibitem{ref54}
					R. Zi,	Z. Zhang,  Stability threshold of Couette flow for 2D Boussinesq
						equations in Sobolev spaces, J. Math. Pures Appl. 179 (9) (2023) 123-182.
					
					
						
					
						
						\bibitem{ref10}C. Zillinger, On enhanced dissipation for the Boussinesq equations, J. Differ. Equ. 282 (2021)
						407–445.
						\bibitem{ref11}C. Zillinger, On the Boussinesq equations with non-monotone temperature profiles, J. Nonlinear
						Sci. 31 (2021) 64.
						\bibitem{ref12}C. Zillinger, On echo chains in the linearized Boussinesq equations around traveling waves, arXiv:2103.15441v2, 2021.
						
					
						
						
						
						
						
						
						
						
						
						
						
						
						
						
						
						
						
						
						
						
						
					\end{thebibliography}
				\end{document}